\documentclass[12pt,leqno]{amsart}
\usepackage{a4wide}
\usepackage{amsmath}
\usepackage[italian, english]{babel}
\usepackage{amsfonts}
\usepackage{amssymb}
\usepackage{dsfont}
\usepackage{graphicx}
\usepackage{fancyhdr}
\usepackage{amsthm}
\usepackage{empheq}
\usepackage{cases}
\usepackage[all]{xy}
\usepackage{stmaryrd}
\usepackage[colorlinks, citecolor=blue, linkcolor=red]{hyperref}
\usepackage{mathrsfs}

\usepackage{mathtools}

\usepackage{eucal}

\DeclareMathAlphabet\mathbfcal{OMS}{cmsy}{b}{n}

\def\SF{\mathcal{SF}}
\def\LS{\mathcal{LS}}
\def\PR{\mathcal{PR}}
\def\LSE{\mathcal{LSE}}
\def\E{\mathcal{E}}
\def\HC{\mathcal{HC}}
\def\Y{\mathcal{Y}}

\def\SFf{\mathcal{SF}_{f}}
\def\LSf{\mathcal{LS}_{f}}
\def\PRf{\mathcal{PR}_{f}}
\def\LSEf{\mathcal{LSE}_{f}}
\def\Ef{\mathcal{E}_{f}}
\def\HCf{\mathcal{HC}_{f}}
\def\HCfl{\mathcal{HC}_{f}^{\lambda}}
\def\Yf{\mathcal{Y}_{f}}

\def\SFX{\mathcal{SF}_{X}}
\def\LSX{\mathcal{LS}_{X}}
\def\PRX{\mathcal{PR}_{X}}
\def\LSEX{\mathcal{LSE}_{X}}
\def\EX{\mathcal{E}_{X}}
\def\HCX{\mathcal{HC}_{X}}
\def\HCXl{\mathcal{HC}_{X}^{\lambda}}

\def\YX{\mathcal{Y}_{X}}

\def\SS{{{\mathbb S}}}

\def\RR{{\mathbb R}}

\newcommand{\varrg}{(M, g)}

\newcommand{\erre}{\mathds{R}}

\newcommand{\cinf}{C^{\infty}(M)}

\newcommand{\ricc}{\operatorname{Ric}}
\newcommand{\riem}{\operatorname{Riem}}
\newcommand{\diver}{\operatorname{div}}

\def\rd{\overset{\circ}{Ric}}
\def\rdc{\overset{\circ}{R}}

\newcommand{\ra}{\rightarrow}

\newcommand{\pa}[1]{{\left(#1\right)}}                  
\newcommand{\sq}[1]{{\left[#1\right]}}                  
\newcommand{\abs}[1]{{\left|#1\right|}}                 

\newtheorem{ackn}{Acknowledgments\!\!}

\newtheorem{theorem}{\textbf{Theorem}}[section]
\newtheorem{lemma}[theorem]{\textbf{Lemma}}
\newtheorem{proposition}[theorem]{\textbf{Proposition}}
\newtheorem{cor}[theorem]{\textbf{Corollary}}
\newtheorem{defi}[theorem]{\textbf{Definition}}

\theoremstyle{remark}

\numberwithin{equation}{section}

\title[A potential generalization of some canonical Riemannian metrics]
{A potential generalization\\of some canonical Riemannian metrics}

\author[Giovanni Catino]{Giovanni Catino}
\address[Giovanni Catino]{Dipartimento di Matematica, Politecnico di Milano, Piazza Leonardo da Vinci 32, 20133 Milano, Italy}
\email[]{giovanni.catino@polimi.it}

\author[Paolo Mastrolia]{Paolo Mastrolia}
\address[Paolo Mastrolia]{Dipartimento di Matematica, Universit\`{a} degli Studi di Milano, Via Saldini, Italy.}
\email[]{paolo.mastrolia@unimi.it}



\begin{document}


\begin{abstract}
The aim of this paper is to study new classes of Riemannian manifolds endowed with a smooth potential function, including in a general framework classical canonical structures such as Einstein, harmonic curvature and Yamabe metrics, and, above all, gradient Ricci solitons. For the most rigid cases we give a complete classification, while for the others we provide rigidity and obstruction results, characterizations and nontrivial examples. In the final part of the paper we also describe the ``nongradient'' version of this construction.

%
%
%
%
%
%
%
%

\end{abstract}

\maketitle

%
%
%
%
%

\section{Introduction}

Let $(M,g)$ be a $n$-dimensional, $n\geq 3$, smooth Riemannian manifold with metric $g$. It is well known that the geometry of $(M,g)$ is encoded in its Riemann curvature tensor $\riem$. Since $\riem$ is a quite involved $4$-tensor depending on $g$ (and on the choice of a ``compatible'' connection $\nabla$), it is natural to define and study some {\em canonical} metrics satisfying, in a suitable sense, a {\em simple} curvature condition. Typically, there are two possible approaches, the {\em algebraic} and the {\em analytic} one.

In the first case, one imposes the constancy of $\riem$, or of its algebraic traces, namely the Ricci curvature $\ricc$ and the scalar curvature $R$. To be more precise and to fix the notation, we say that $(M,g)\in\SF$ ({\em space form}), $(M,g)\in\E$ ({\em Einstein manifold}) or $(M,g)\in\Y$ ({\em Yamabe metric}), if, for some $\lambda\in\RR$, the Riemannian metric $g$ on $M$ satisfies
$$
\riem = \frac{\lambda}{2(n-1)}\, g \owedge g \,, \quad\quad \ricc = \lambda\, g \,, \quad\quad
R = n \lambda \,,
$$
respectively. Here, and in the rest of the paper, $\owedge$ denotes the standard Kulkarni-Nomizu product of symmetric $2$-tensors. Clearly, the three classes of Riemannian manifolds introduced above satisfy
$$
\begin{array}{ccccc}
\SF& \subset& \E &\subset &\Y
\end{array}
$$
and it is well known that, in dimension $n=3$, $\SF =\E$.

On the other hand, from the analytic point of view, the aim is to simplify the curvature by imposing some  differential condition. A quite natural and not too restrictive way to do this is to consider curvature tensors belonging to the kernel of a first order linear differential operator. Some well known conditions of this type can be given by saying that $(M, g)$ belongs to
\begin{itemize}
\item $\LS$ if $\nabla(\riem)=0$ ({\em locally symmetric metrics});
\item $\PR$ if $\nabla(\ricc)=0$ ({\em metrics with parallel Ricci curvature});
\item $\HC$ if $\diver(\riem)=0$ ({\em harmonic curvature metric}).
 \end{itemize}
 Note that, by Bianchi identities, we can redefine the class $\Y$ of Yamabe metrics using the condition $\diver(\ricc)=0$ or, equivalently, $\nabla R=0$.
Here and in the rest of the paper $\diver$ denotes the divergence operator (see Section \ref{sec-def} for the definition). Obviously, $\SF\subset\LS\subset\PR$ and, by Bianchi identities, $\PR\subset\HC$, $\E\subset\HC\subset \Y$. Thus, we have the inclusions
\begin{equation}\label{chain}
\begin{array}{ccccccccc}
\,& \,& \LS & \subset & \PR & \, & \, & \, \\
\,& \rotatebox{45}{$\subset$}& \cup & \, & \cup & \rotatebox{45}{$\cap$} &\, & \,\\
\SF& \subset& \LSE & \subset& \E &\subset & \HC & \subset &\Y
\end{array}
\end{equation}
where, by definition, $\LSE:=\LS\cap\E$ ({\em locally symmetric Einstein metrics}). Note that, in dimension $n\geq 4$, all the inclusions are strict.

This classes of metrics certainly do not exhaust {\em all} the possible canonical metrics on a Riemannian manifold: our choice is essentially made in such a way that Einstein metrics (and Ricci solitons, as we will see) are the ``cornerstone'' of our construction, and the conditions that we impose are consequently focused on the Ricci tensor. We note that, in principle, one could also consider ``higher order'' conditions, such as $\nabla^{k}\riem =0$ or $\nabla^{k}\ricc=0$, $k\geq 2$, but  these relations give rise again to $\LS$ and $\PR$, respectively, by the results in \cite{nomoze, tanno}. However, one can consider other higher order analytic curvature conditions in order to generalize locally symmetric metrics, such as, for instance, the class of semi-symmetric spaces introduced by Cartan in \cite{CartanBook}.

The class $\SF$ is the most rigid, since, up to quotients, it contains only $\SS^{n}$, $\RR^{n}$ and $\mathbb{H}^{n}$ with their standard metrics. Locally symmetric spaces $\LS$ were classified by Cartan \cite{cartan}, while, from the de Rham decomposition theorem (\cite{besse}), $\PR$ metrics are locally Riemannian products of Einstein metrics. On the other hand, given any compact manifold $M$, there always exists a Riemannian metric $g$ such that $(M,g)\in\Y$ (see e.g. \cite{leepar}). The remaining classes are more flexible. In particular $\E$ and $\HC$, in the last decades, have been studied by many researchers, also for their connections with Physics in General Relativity and Yang-Mills Theory. In fact, these metrics arise naturally as solutions of the Euler-Lagrange equations of some variational problems. More precisely, in dimension $n\geq 3$, the class $\E$ of Einstein metrics coincides with the set of critical points of the Einstein-Hilbert functional
$$
\mathbf{S}(g):=\int_{M} R \,dV_{g}
$$
on the space of volume one metrics, while the $\HC$ equation arises in studying in a given Riemannian vector bundle $\pi:E\rightarrow M$ critical metric connections $\nabla$ for
the Yang-Mills functional
$$
\mathbf{YM}(\nabla) := \frac{1}{2} \int_{M}|\mathrm{R}^{\nabla}|^{2} dV_{g}\,,
$$
where $\mathrm{R}^{\nabla}$ is the curvature of the connection $\nabla$. Yang-Mills connections are characterized by $d^{*}{\mathrm{R}}^{\nabla}=0$, where $d^{*}$ is the formal adjoint (with respect to the standard volume form $dV_{g}$) of the exterior differential $d$ acting on $E$-valued differential forms on $(M,g)$ (see e.g. \cite{derd2}). Note that $d^{*}$ becomes the ordinary divergence operator $\diver$ when $E=TM$ and $\nabla$ is the Levi-Civita connection of $g$. In view of the Bianchi identity $d\mathrm{R}^{\nabla}=0$, this means that the curvature of any Yang-Mills connection is {\em harmonic} with respect to the standard Hodge Laplacian $\Delta^{H}:=d d^{*} + d^{*}d$, acting on two forms.

The aforementioned canonical metric structures, which  have been the subject of extensive investigations in the last decades and  are by now considered ``classical'', can be thought as solutions of PDEs of the form
$\mathfrak{F}[g] = 0$,
where $\mathfrak{F}$ is a differential operator acting on the metric $g$. The related literature is enormous, and we don't even try to give here a comprehensive bibliography: the interested reader can consult for instance the well-known \cite{besse} and references therein.

\

In  recent years many mathematicians have focused their research on more general structures, considering particular conditions that involve the curvature of a metric  and a \emph{potential}, that is, a smooth function defined on the underlying manifold (metric measure spaces, conformal invariants, Einstein-type manifols, dilaton fields, etc.) In this situation, it is  natural to study solutions $(g, f)$, with $f\in \cinf$, of  $\mathfrak{F}[g, f] = 0$, where $\mathfrak{F}$ is again a differential operator now acting on the metric $g$ and on the potential $f$.
A particularly important example arises from the pioneering works of Hamilton \cite{hamilton1} and Perelman \cite{perelman} towards the solution of the Poincar\'e conjecture in dimension three: indeed, with their seminal papers they have generated a flourishing activity in the research of self-similar solutions, or solitons, of the Ricci flow. From the static point of view, these structures are characterized by the condition
$$
\ricc_{f}:=\ricc+\nabla^{2} f = \lambda\,g \,,
$$
where $\ricc_{f}$ is the {\em Bakry-Emery Ricci tensor}, $f\in C^{\infty}(M)$ is called the {\em potential}, $\lambda \in \RR$ and $\nabla^2$ is the Hessian. In this case, we say that $(M,g,f)\in\Ef$ ({\em gradient Ricci soliton}). It is apparent that this is a reasonable generalization of the Einstein condition which, interpreted as a global prescription on the Ricci curvature of $g$, was firstly considered by Lichnerowicz (see e.g. \cite{bou}). In particular, if $(M, g)\in\E$ then $(M, g, f=c\in\RR)\in \Ef$, and we can add another inclusion to the previous diagram as follows:
\begin{equation*}
\begin{array}{ccccccccc}
\,& \,& \LS & \subset & \PR & \, & \, & \, \\
\,& \rotatebox{45}{$\subset$}& \cup & \, & \cup &\rotatebox{45}{$\cap$} &\, & \,\\
\SF	&\subset& \LSE & \subset	&\E	&\subset	&\HC	&\subset	&\Y\\
\,	&\,	&\, &\,		& \cap	&\,			& \,		&\,		&\, \\
\,	&\, &\, &\, 	&\Ef	&\,	&\,	&\,	&\,
\end{array}
\end{equation*}
The main aim of this paper is to propose a ``potential'' generalization of the previous framework, that is, we introduce and begin to study new classes of privileged metrics $g$ on Riemannian manifolds $M$ endowed with smooth potentials function $f$, which extend the diagram above. We first give the following

\begin{defi}\label{ladefinizione} Let $(M,g)$ be a $n$-dimensional, $n\geq 3$, Riemannian manifold with metric $g$. We say that the triple $(M,g,f)$ belongs to
\begin{itemize}

\item $\SFf$ ({\em $f$-space forms}) if there exist $f\in\cinf$ and $\lambda\in\RR$ such that
$$
\riem_{f} := \riem + \frac{1}{n-2}\Big(\nabla^{2}f-\frac{\Delta f}{2(n-1)}\,g\Big)\owedge g = \frac{\lambda}{2(n-1)} g \owedge g\,;
$$

\item $\LSEf$ ({\em $f$-locally symmetric Einstein metrics}) if there exist $f\in\cinf$ and $\lambda\in\RR$ such that
$$
\nabla \big(\riem_{f}\big) =0 \quad\hbox{and}\quad \ricc_{f} = \lambda g\,;
$$

\item $\Ef$ ({\em gradient Ricci solitons}) if there exist $f\in\cinf$ and $\lambda\in\RR$ such that
$$
\ricc_{f} = \lambda\,g \,;
$$

\item $\HCf$ ({\em $f$-harmonic curvature metrics}) if there exists $f\in\cinf$  such that
$$
\diver \big(e^{-f}\riem\big) =0 \,;
$$

\item $\Yf$ ({\em $f$-Yamabe metrics}) if there exists $f\in\cinf$  such that
$$
\diver \big(e^{-f}\ricc\big) = 0\,\hbox{, i.e.}\quad \nabla R = 2 \ricc (\nabla f,\cdot) \,.
$$
\end{itemize}
Moreover, we say that $(M,g,f)$ belongs to
\begin{itemize}

\item $\LSf$ ({\em $f$-locally symmetric metrics}) if there exists $f\in\cinf$ such that
$$
\nabla \big(\riem_{f}\big) =0 \,;
$$

\item $\PRf$ ({\em metrics with parallel Bakry-Emery Ricci tensor}) if there exists $f\in\cinf$ such that
$$
\nabla \big(\ricc_{f}\big) =0  \,.
$$
\end{itemize}
\end{defi}

Note that we recover the corresponding sets in \eqref{chain} when $\nabla f=0$ on $M$; in this latter case, we say that the structure is {\em trivial}. In particular, some  computations (see Section \ref{sec-frame}) show that
\begin{equation}\label{cattedrale}
\begin{array}{ccccccccc}
\,& \,& \LS & \subset & \PR & \, & \, & \, \\
\,& \rotatebox{45}{$\subset$}& \cup & \, & \cup &\rotatebox{45}{$\cap$} &\, & \,\\
\SF	&\subset &\LSE &\subset	&\E	&\subset	&\HC	&\subset	&\Y\\
\cap	&\,			& \cap	&\,			& \cap		&\,		&\cap  & \, &\cap\\
\SFf	&\subset &\LSEf	&\subset &\Ef	&\subset	&\HCf	&\subset	&\Yf\\
\,& \rotatebox{-45}{$\subset$}& \cap & \, & \cap &\rotatebox{-45}{$\cup$} &\, & \,\\
\,& \,& \LSf & \subset & \PRf & \, & \, & \, \\
\end{array}
\end{equation}

{\bf Remarks:}
\begin{enumerate}

\item[1.]  We observe that, with the exception of $\HCf$ and $\Yf$, all the classes introduced in Definition \ref{ladefinizione} represent Riemannian metrics for which the associated ``$f$-curvatures'' ($\riem_{f}$ and $\ricc_{f}$) satisfies simple algebraic/analytic conditions. On the other hand,  to define the classes $\HCf$ and $\Yf$, we impose the vanishing of the divergence of the ``weighted'' tensors $e^{-f}\riem$ and $e^{-f}\ricc$ instead of considering the apparently natural relations
$$
\diver\big(\riem_{f})=0 \quad\hbox{and}\quad \diver\big(\ricc_{f}\big)=0 \,.
$$
In fact, it turns out that these latter are not good candidates since, for instance, gradient Ricci solitons ($\Ef$) satisfy the second but, in general, not the first condition. To clarify this apparent discrepancy in Definition \ref{ladefinizione}, in Section \ref{sec-frame} we prove equivalent conditions characterizing these classes showing, in particular, that $\HCf$ and $\Yf$ can be defined (in a precise way) by means of the Bakry-Emery Ricci tensor $\ricc_{f}$, giving to this latter a prominent role. This is perfectly reasonable, since the equation $\diver(\riem)=0$, defining $\HC$, is, as a matter of fact, a condition on $\ricc$.

\item[2.] As we have already observed, gradient Ricci solitons, besides being important in Ricci flow theory, represent a natural generalization of Einstein metrics: the symmetric $2$-tensor $\nabla^{2}f$, the Hessian of the potential $f$, measures how much the manifold deviates from being Einstein and the Bakry-Emery Ricci tensor $\ricc_{f}$ replaces $\ricc$. On the other hand, the ``trace part'' of the curvature tensor is given by $\frac{1}{n-2}\operatorname{A} \owedge \,g$, where $\operatorname{A}$ is the Schouten tensor $A:=\ricc-\frac{R}{2(n-1)}g$. It is then natural to consider a corresponding generalization of the Riemann tensor, $\riem_{f}$, adding to $\riem$ the $4$-tensor
$$
\frac{1}{n-2}\Big(\nabla^{2}f-\frac{\Delta f}{2(n-1)}\,g\Big)\owedge g\,.
$$

\item[3.] The equation of gradient Ricci solitons ($\Ef$) can be obtained by tracing the one defining $\SFf$. Thus, in principle, we could have introduced $f$-Yamabe metrics \emph{via} algebraic simplification by tracing the $\Ef$ equation, obtaining
$$
R_{f}:=R + \Delta f = n \lambda
$$
for some $\lambda\in\RR$. We know that this condition alone (if not coupled with other constraints, see Definition \ref{def-hcfl} below) is too ``weak'' to define a meaningful set of metrics, since, for instance, on every compact Riemannian manifold $(M,g)$ one can always find a smooth function $f$ solving this equation for a suitable $\lambda\in\RR$. On the other hand, thinking of it as a prescribed scalar curvature problem, given any function $f\in C^{\infty}(M)$, we could always find a solution (i.e. a metric) if $\lambda\leq 0$, or if $\lambda>0$ and $M$ admits a metric with positive scalar curvature (see the seminal works of Kazdan and Warner \cite[Theorem 6.4]{kazwar2}).

\item[4.] It is well known that compact gradient shrinking, steady and expanding Ricci solitons $\Ef$ can be characterized as critical points of the $\mathbf{F}$ and $\mathbf{W}$, $\mathbf{W}_{-}$ functionals, respectively (see e.g. \cite{caorev}). On the other hand,
the class $\HCf$ arises naturally in studying critical metric connections $\nabla$ in a given Riemannian vector bundle $\pi:E\rightarrow M$ for the ``weighted'' Yang-Mills functional
$$
\mathbf{YM}_{f}(\nabla) := \frac{1}{2} \int_{M}|\mathrm{R}^{\nabla}|^{2} e^{-f}dV_{g} \,,
$$
that leads to the so called Yang-Mills-Dilaton field theory. A simple computation, following the one for $\mathbf{YM}$ (see e.g. \cite{boulaw}), shows that weighted Yang-Mills connections are characterized by $d_{f}^{*}\mathrm{R}^{\nabla}=0$, where $d_{f}^{*}$ is the formal adjoint of the exterior differential $d$ with respect to the weighted volume form $e^{-f}dV_{g}$ (see \cite{bue}). Note that $d^{*}_{f}$ becomes the $f$-divergence operator $e^{f}\diver(e^{-f}\,)$ when $E=TM$ and $\nabla$ is the Levi-Civita connection of $g$. By Bianchi identity $d\mathrm{R}^{\nabla}=0$, this means that the curvature of any weighted Yang-Mills connection is {\em weighted harmonic} with respect to the weighted Hodge Laplacian
$$
\Delta^{H}_{f}:=d d^{*}_{f} + d^{*}_{f}d \,.
$$

\item[5.] In our discussion we have so far considered only the case of dimension greater than three. We observe that in dimension $n=2$,  the geometry of a Riemann surface $(M,g)$ is contained in the scalar curvature $R$. In particular, $\ricc=\frac{R}{2}g$ and the equation defining $\Yf$ yields
$$
\nabla\Big(e^{-f}R\Big) = 0 \quad \Longleftrightarrow \quad R = C e^{f} \,,
$$
for some $C\in\RR$. This is equivalent to the classical problem of prescribing (with sign) the Gauss (scalar) curvature of a Riemann surface. By the seminal works of Kazdan and Warner \cite{kazwar2}, it follows that, on a compact surface $M$, given {\em any} smooth function $f$, there exists a Riemannian metric $g$ such that $(M,g,f)\in \Yf$ (in the zero genus case, a solution is the scalar flat metric).

\item[6.] We will see that, as one can expect, the classes $\SFf$, $\LSf$, $\LSEf$ and $\PRf$ do not differ too much from their classical counterparts, as we will show in Propositions \ref{pro-sff} and \ref{pro-lsf}; however, they still contain some interesting Riemannian spaces, such as generalized cylinders (with Gaussian potential) and the Bryant soliton.

\end{enumerate}

The paper is organized in the following sections:

\tableofcontents

\

\section{Main results} \label{sec-main}

In this section we present some of the main results of the paper, concerning all of the classes introduced above. To simplify the exposition, we will always assume $(M,g)$ complete, even if clearly not needed in most of the results, and the dimension $n\geq 3$.

We begin with the classification of $f$-space forms. Observe that, in dimension $n=3$, we have $\SFf=\Ef$; in higher dimension $n\geq 4$, in Section \ref{sec-rig} we will prove the following

\begin{proposition}\label{pro-sff} Let $(M,g,f)\in\SFf$, then
\begin{itemize}

\item if $\lambda>0$,  $(M,g,f)$ is isometric, up to quotients, to either $\big(\SS^{n},g_{\SS^{n}}, f=c\in\RR\big)$, $\big(\RR\times\SS^{n-1}, dr^{2}+g_{\SS^{n-1}}, f=\frac{\lambda}{2} r^{2}\big)$ or to $\big(\RR^{n},g_{\RR^{n}}, f=\frac{\lambda}{2}|x|^{2}\big)$;

\item if $\lambda=0$,  $(M,g)$ is isometric, up to quotients, to either $\big(\RR^{n},g_{\RR^{n}}, f=c\in\RR\big)$ or the Bryant soliton.

\item if $\lambda<0$, around any regular point of $f$ the manifold $(M,g)$ is locally a warped product with codimension one fibers of constant sectional curvature. Moreover, if the Ricci curvature is nonnegative, $(M,g)$ is rotationally symmetric.
\end{itemize}
\end{proposition}
We recall that the Bryant soliton, constructed in \cite{bry}, is the unique (up to homotheties) rotationally symmetric gradient steady Ricci solitons with positive sectional curvature.

\

As far as the classes $\LSf$ and $\LSEf$ are concerned, note that, in dimension $n=3$, $\LSf=\PRf$ and $\LSEf=\Ef$; in higher dimension $n\geq 4$, again in Section \ref{sec-rig}, we prove

\begin{proposition}\label{pro-lsf} If $(M,g,f)\in\LSf$, then $(M,g,f)\in \LS \cup \SFf$. Furthermore, if $(M,g,f)\in \LSEf$, then either $(M,g,f)\in \LSE \cup \SFf$ or it is isometric, up to quotients, to a Riemannian product $\big(\RR^{k}\times N, g_{\RR^{k}}+g_{N}, f=\frac{\lambda}{2}|x|^{2}_{k}\big)$, $k\geq 1$, with $N\in\LSE$ being a $(n-k)$-dimensional locally symmetric Einstein manifold.
\end{proposition}

The previous results are a consequence of the fact that the equations defining $f$-space forms and $f$-locally symmetric metrics imply strong conditions on the Weyl tensor $W$, as we will see in Section \ref{sec-frame}, since they involve the full $f$-curvature tensor $\riem_{f}$. On the other hand, when one imposes conditions only on $\ricc_{f}$, that is on the trace part of $\riem_{f}$, it is reasonable to expect rigidity only assuming further conditions on the traceless part, i.e. $W$. The next theorem extends to the $\HCf$ class the well known result concerning the local structure of locally conformally flat gradient Ricci solitons.

\begin{theorem}\label{teo-lcf}
Let $(M,g,f)\in \HCf$. If $(M,g)$ is locally conformally flat, then, around any regular point of $f$, it is locally a warped product with codimension one fibers of constant sectional curvature.
\end{theorem}

It is well known that compact locally conformally flat gradient Ricci solitons have constant curvature (see e.g. \cite{emilanman}). We will see that such a conclusion cannot be extended to manifolds in $\HCf$, since we can construct rotationally symmetric examples on $\mathbb{S}^{1}\times\mathbb{S}^{n-1}$ (see Section \ref{sec-hcf}).

In order to state the next results, we first recall that, as we have already observed, $\HC\subset \Y$, i.e. harmonic curvature metrics have constant scalar curvature. This is not true in general for the potential counterpart $\HCf$, but, for instance, on gradient Ricci solitons it holds that $R_f = R+\Delta f =n\lambda$. Thus, it is natural to introduce the following
\begin{defi}\label{def-hcfl} Let $(M,g,f)$ be a $n$-dimensional manifold with Riemannian metric $g$ and $f\in C^{\infty}(M)$. We say that $(M,g,f)\in \HCfl$ if $(M,g,f)\in \HCf$ and, for some $\lambda\in\RR$, $R_{f}:=R+\Delta f= n \lambda$.
\end{defi}
Note that $\Ef\subset\HCfl\subset\HCf$ and also, by a simple computation, $\PRf\subset\HCfl$. We will see in a short while that the class $\HCfl$ (and $\HCf$, in some cases) coincides with $\Ef$ under some additional conditions. First, we recall that in dimension four, under the topological condition $\tau(M)\neq 0$, Bourguignon in \cite{bou2tr} proved that $\HC=\E$ (where $\tau$ is the signature of $M$). Moreover, the classical Hirzebruch signature formula says that
$$
48\pi^{2} \tau(M)= \int_{M} |W^{+}|^{2} - \int_{M}|W^{-}|^{2} \,,
$$
where $W^+$ and $W^-$ are the self-dual and anti-self-dual parts of the tensor $W$, respectively.
In the next theorem  we extend Bourguignon's result in the $\HCfl$ case, and, more generally, in the $\HCf$ case, under an additional regularity assumption (which is automatically satisfied by $\HC$ metrics, as proved in \cite{deturck}).
\begin{theorem}\label{teo-4d} Let $M$ be a four dimensional compact  manifold with $\tau(M)\neq 0$. Then,
\begin{itemize}
\item[i)] $(M,g,f)\in \HCf$ and, in harmonic coordinates, $g$ and $f$ are real analytic if and only if $(M,g,f) \in \Ef$.
\item[ii)] $(M,g,f)\in \HCfl$  if and only if $(M,g,f) \in\Ef$.
\end{itemize}
\end{theorem}
Note that gradient Ricci solitons satisfy the analyticity assumption, but we do not know in general if this is true for metric in $\HCf$.

We recall that a metric is half conformally flat if it is self-dual or anti-self-dual, namely if $W^{-} = 0$ or $W^{+} = 0$, respectively (see \cite[chapter 13, section C]{besse} for a nice overview).
As a simple consequence of Theorem \ref{teo-4d} we have
\begin{cor}\label{cor-hcf} Let $M$ be a four dimensional compact  manifold and let $(M,g,f)\in\HCfl$. If $(M,g)$ is half conformally flat but not conformally flat, then
\begin{itemize}
\item[i)] if $\lambda>0$, $(M,g)$ is isometric to $\mathbb{CP}^{2}$ with its canonical metric;

\item[ii)] if $\lambda=0$, the universal covering of $(M,g)$ is isometric to a $K3$ surface with the Calabi--Yau metric;

\item[iii)] if $\lambda<0$, $(M,g)\in \E$ with negative scalar curvature.
\end{itemize}

\end{cor}

In general dimension $n\geq 3$ we can prove, assuming positive sectional curvature, the following extension of a Berger result (see \cite{besse}).
\begin{proposition}\label{pro-sec} Let $(M,g)$ be a $n$-dimensional, $n\geq 3$, compact manifold with positive sectional curvature. Then $(M,g,f)\in\HCfl$ if and only if $(M,g,f) \in \Ef$.
\end{proposition}
A classical result by Tachibana (\cite{tachi}) says that if $(M, g)\in \HC$, with positive curvature operator, then $(M, g)$ is, up to quotients, the round sphere; in the $\HCfl$ case we have

\begin{cor}\label{cor-tac}  Let $(M,g)$ be a $n$-dimensional, $n\geq 3$ compact manifold with positive curvature operator. If $(M,g,f) \in \HCfl$ , then $f$ is constant and $(M,g)$ is isometric, up to quotients, to $\SS^{n}$.
\end{cor}

Finally, in Section \ref{sec-hcf}, following Derdzinski (\cite{derd3}) we construct a family of compact Riemannian manifolds in $\HCf$, which are not gradient Ricci solitons; we also  exhibit an explicit noncompact example.

As far as the class $\Yf$ is concerned, in Section \ref{sec-yf} we construct another family of examples and we also prove an obstruction result to the existence of $f$-Yamabe metrics in a given conformal class, in the same spirit of the classical work of Kazdan and Warner (\cite{kazwar1}) concerning the prescribed scalar curvature problem.
Note that, in dimension $2$, this connection has already been observed in the Introduction.
In the particular case of the sphere, the obstruction reads as
\begin{proposition}\label{pro-sphere} If $f \in C^{\infty}(\SS^{n})$ is a first  spherical harmonic on the round sphere $(\SS^{n},g_0)$, then there are no conformal metrics $g\in[g_{0}]$ such that $(M,g,f)\in\Yf$.
\end{proposition}
It is interesting to note that the same functions $f$ on $\SS^{n}$ (spherical harmonics) give obstructions  in specifying (conformally) the gradient of the scalar curvature in two different ways: $\nabla R = \nabla f$  (i.e. prescribed scalar curvature, $R=f$ up to constants) and $\nabla R = 2\ricc(\nabla f)$ (i.e., $f$-Yamabe metrics).

\

\section{Definitions and some useful formulas} \label{sec-def}

In this section we collect some useful definitions and properties of various geometric tensors, and fix our conventions and notation. To perform computations, we freely use the method of the moving frame, referring to a local orthonormal (co)frame of the $n$-dimensional Riemannian manifold $(M,g)$. In some situations we will use $\langle X,Y \rangle$ instead of $g(X,Y)$, for $X,Y\in\mathfrak{X}(M)$. We also fix the index range $1\leq i, j, \ldots \leq n$ and we recall that the Einstein convention of summing over the repeated indices will be adopted throughout the article.
\subsection{General definitions}
The $(1 ,3)$-Riemann curvature tensor of a Riemannian manifold $(M,g)$ is defined as
$$
\mathrm{R}(X,Y)Z=\nabla_{X}\nabla_{Y}Z-\nabla_{Y}\nabla_{X}Z-\nabla_{[X,Y]}Z\,.
$$
 In coordinates we have
$R^{l}_{ijk}\tfrac{\partial}{\partial
  x^{l}}=\mathrm{R}\big(\tfrac{\partial}{\partial
  x^{j}},\tfrac{\partial}{\partial
  x^{k}}\big)\tfrac{\partial}{\partial x^{i}}$ and we denote by
$R_{ijkl}=\delta_{im}R^{m}_{jkl}$ its $(0, 4)$-version that we call $\riem$. The Ricci tensor $\ricc$ is obtained by the contraction
$R_{ik}=\delta^{jl}R_{ijkl}$ and $R=\delta^{ik}R_{ik}$ will
denote the scalar curvature. We recall that, in dimension $n=2$, all the geometry of the manifold is encoded in the scalar curvature, since $\ricc=\frac{R}{2}g$.

The so called Weyl tensor is 
defined by the following decomposition formula (see \cite[Chapter~3,
Section~K]{gahula}) in dimension $n\geq 3$,
\begin{eqnarray}
\label{Weyl}
W_{ijkl}  & = & R_{ijkl} \, - \, \frac{1}{n-2} \, (R_{ik}\delta_{jl}-R_{il}\delta_{jk}
+R_{jl}\delta_{ik}-R_{jk}\delta_{il})  \nonumber \\
&&\,+\frac{R}{(n-1)(n-2)} \,
(\delta_{ik}\delta_{jl}-\delta_{il}\delta_{jk})\, \, .
\end{eqnarray}
The Weyl tensor shares the symmetries of the curvature
tensor. Moreover, as it can be easily seen by the formula above, all of its contractions with the metric are zero, i.e. $W$ is totally trace-free. In dimension three, $W$ is identically zero on every Riemannian manifold, whereas, when $n\geq 4$, the vanishing of the Weyl tensor is
a relevant condition, since it is  equivalent to the local
  conformal flatness of $(M,g)$. We also recall that in dimension $n=3$,  local conformal
  flatness is equivalent to the vanishing of the Cotton tensor
\begin{equation}\label{def_cot}
C_{ijk} =  R_{ij,k} - R_{ik,j}  -
\frac{1}{2(n-1)}  \big( R_k  \delta_{ij} -  R_j
\delta_{ik} \big)\,,
\end{equation}
where $R_{ij,k}=\nabla_k R_{ij}$ and $R_k=\nabla_k R$ denote, respectively, the components of the covariant derivative of the Ricci tensor and of the differential of the scalar curvature.
By direct computation, we can see that the Cotton tensor $C$
satisfies the following symmetries
\begin{equation}\label{CottonSym}
C_{ijk}=-C_{ikj},\,\quad\quad C_{ijk}+C_{jki}+C_{kij}=0\,,
\end{equation}
moreover it is totally trace-free,
\begin{equation}\label{CottonTraces}
C_{iik}=C_{iji}=C_{ikk}=0\,,
\end{equation}
by its skew--symmetry and Schur lemma.  Furthermore, it satisfies
\begin{equation}\label{eq_nulldivcotton}
C_{ijk,i} = 0,
\end{equation}
see for instance \cite[Equation 4.43]{catmasmonrig}. We recall that, for $n\geq 4$,  the Cotton tensor can also be defined as one of the possible divergences of the Weyl tensor:
 \begin{equation}\label{def_Cotton_comp_Weyl}
 C_{ijk}=\pa{\frac{n-2}{n-3}}W_{tikj, t}=-\pa{\frac{n-2}{n-3}}W_{tijk, t}.
 \end{equation}
 A computation shows that the two definitions coincide (see e.g. \cite{alimasrig}).

 The Bach tensor, first introduced in general relativity by Bach, \cite{bac}, is by definition
 \begin{equation}\label{def_Bach_comp}
   B_{ij} = \frac{1}{n-3}W_{ikjl, lk} + \frac{1}{n-2}R_{kl}W_{ikjl} = \frac{1}{n-2}\pa{C_{jik, k}+R_{kl}W_{ikjl}}.
 \end{equation}

  A computation using the commutation rules for the second covariant derivative of the Weyl tensor or of the Schouten tensor (see \cite{catmasmonrig}) shows that the Bach tensor is symmetric (i.e. $B_{ij}=B_{ji}$); it is also evidently trace-free (i.e. $B_{ii}=0$). It is worth reporting here the following interesting formula for the divergence of the Bach tensor (see e.g. \cite{caoche2} for its proof)
\begin{equation}\label{diverBach}
  B_{ij, j} = \frac{n-4}{\pa{n-2}^2}R_{kt}C_{kti}.
\end{equation}

Since we will use in the sequel of the paper, we recall the definition of the Kulkarni-Nomizu product of two symmetric two-tensors $\alpha,\beta$:
$$
\big(\alpha\owedge\beta)_{ijkt} = \alpha_{ik}\beta_{jt}-\alpha_{it}\beta_{jk}+\alpha_{jt}\beta_{ik}-\alpha_{jk}\beta_{it}\,.
$$
In particular, when $\beta=g$, we have the following expression for the divergence of $\alpha \owedge g$:
\begin{equation}\label{eq-divkn}
\big(\alpha\owedge\beta)_{tijk,t} = \alpha_{tj,t}\delta_{ik}-\alpha_{tkt}\delta_{ij}+\alpha_{ik,j}-\alpha_{ij,k}\,.
\end{equation}
Finally, we recall that a Codazzi tensor $T$ is a symmetric $(0, 2)$-tensor satisfying the Codazzi equation
\[
T_{ij, k} = T_{ik, j}.
\]
For a general overview on Codazzi tensors, we refer to \cite[Section 16C]{besse}.
\subsection{Ricci solitons}
We recall here some useful equations satisfied by every gradient Ricci soliton $(M,g, f)\in \Ef$. By definition,
 \begin{equation}\label{def_sol}
   R_{ij}+f_{ij}=\lambda g_{ij}, \quad \lambda \in \erre,
 \end{equation}
where $f_{ij}=\nabla_i\nabla_j f$ are the components of the Hessian of $f$ (see e.g. \cite{emilanman}).
\begin{lemma} Let $(M^n,g)$ be a gradient Ricci soliton of dimension $n\geq 3$. Then the following equations holds:
\begin{equation*}\label{eq_tra}
R_f := R+\Delta f  = n \lambda,
\end{equation*}
\begin{equation*}\label{eq_sch}
\nabla R = 2\ricc(\nabla f, \cdot), \quad \text{i.e. }\quad R_i = 2 f_t R_{it},
\end{equation*}
\begin{equation*}\label{eq_hamide}
R + |\nabla f|^2 = 2\lambda f + c, \quad \hbox{for some}\,\, c\in\RR,
\end{equation*}
\begin{equation*}
  R_{ij, k}-R_{ik, j} = -R_{tijk, t} = -f_tR_{tijk}.
\end{equation*}

\end{lemma}

The tensor $D$, here denoted by $D^{\nabla f}$ to distinguish it from its ``generic'' counterpart $D^{X}$ (see Section \ref{sec-nongrad}), was introduced by Cao and Chen  in \cite{caoche1} and turned out to be a fundamental tool in the study of the geometry of gradient Ricci solitons (more in general for gradient Einstein-type manifolds, see \cite{catmasmonrig}). In components it is defined as
 \begin{align}\label{def_D}
   D^{\nabla f}_{ijk}=&\frac{1}{n-2}\pa{f_kR_{ij}-f_jR_{ik}}+\frac{1}{(n-1)(n-2)}f_t\pa{R_{tk}\delta_{ij}-R_{tj}\delta_{ik}}\\\nonumber
 &\,-\frac{R}{(n-1)(n-2)}\pa{f_k \delta_{ij}-f_j \delta_{ik}}.
 \end{align}
 The $D^{\nabla f}$ tensor is skew-symmetric in the second and third indices (i.e. $D^{\nabla f}_{ijk}=-D^{\nabla f}_{ikj}$) and totally trace-free (i.e. $D^{\nabla f}_{iik}=D^{\nabla f}_{iki}=D^{\nabla f}_{kii}=0$).
Note that our convention for the tensor $D$ differs from that in \cite{caoche1}.

\
%

If $(M, g, X)$ is a Ricci soliton structure on $\varrg$, with $X\in\mathfrak{X}(M)$, the defining equation becomes
\[
R_{ij} +\frac{1}{2}(X_{ij}+X_{ji})=\lambda \delta_{ij}.
\]
Moreover we have (see \cite{catmasmonrig})
\begin{eqnarray*}
&R_X := R +  \diver(X)=n\lambda; \\
&\nabla R = 2\ricc(X, \cdot)+\diver\pa{\mathcal{A}^X}, \quad \text{i.e. }\quad R_i = 2 X_t R_{it}+ X_{it, t}-X_{ti, t},
\end{eqnarray*}
where $\mathcal{A}^X$ is the antisymmetric part of the covariant derivative of $X$; in components, $(\mathcal{A}^X)_{ij}=X_{ij}-X_{ji}$.

Finally, we recall the following formula due to B\"{o}chner, \cite{BocYan}, and rediscovered many times in recent years.

\begin{lemma}\label{lem-boch}
Let $X$ be a vector field on the Riemannian manifold
$(M, g)$. Then
\begin{equation*}
\diver\pa{\mathcal{L}_Xg}(X)=\frac{1}{2}\Delta|X|^2-|\nabla X|^2+\ricc(X,X)+\nabla_X(\diver X) \,,
\end{equation*}
or in coordinates
\begin{equation*}
\big(X_{iji}+X_{jii}\big)X_{j}=\frac{1}{2}\Delta|X|^2-|\nabla X|^2+R_{ij}X_{i}X_{j}+X_{jji}X_{i} \,.
\end{equation*}

\end{lemma}

\

\section{Canonical metrics revisited: equivalent conditions} \label{sec-frame}

\Small{

\begin{equation*}
\begin{array}{ccccccccc}
\,& \,& \LS & \subset & \PR & \, & \, & \, \\
\,& \rotatebox{45}{$\subset$}& \cup & \, & \cup &\rotatebox{45}{$\cap$} &\, & \,\\
\SF	&\subset &\LSE &\subset	&\E	&\subset	&\HC	&\subset	&\Y\\
\cap	&\,			& \cap	&\,			& \cap		&\,		&\cap  & \, &\cap\\
\SFf	&\subset &\LSEf	&\subset &\Ef	&\subset	&\HCf	&\subset	&\Yf\\
\,& \rotatebox{-45}{$\subset$}& \cap & \, & \cap &\rotatebox{-45}{$\cup$} &\, & \,\\
\,& \,& \LSf & \subset & \PRf & \, & \, & \, \\
\end{array}
\end{equation*}

}

\

\

\normalsize

The aim of this section is to present equivalent conditions characterizing some of the classes in Definition \eqref{ladefinizione}; for the sake of completeness and to highlight the similarities and the differences with the ``potential'' counterpart, we report the well known characterizations of the classical structures.

Here $(M, g)$ will be a smooth Riemannian manifold of dimension  $n\geq3$ with metric $g$. First we  recall that the decomposition in \eqref{Weyl} can be globally (and orthogonally) written, using the Schouten tensor $\operatorname{A} = \ricc -\frac{R}{2(n-1)}g$, as
\begin{equation}
\operatorname{Riem} = W +\frac{1}{n-2} \operatorname{A} \owedge g.
\end{equation}
It this then natural to introduce a new tensor, that we call $\operatorname{A}_f$ (the \emph{$f$-Schouten} tensor), in such a way that
\[
 \riem_{f} := \riem + \frac{1}{n-2}\Big(\nabla^{2}f-\frac{\Delta f}{2(n-1)}\,g\Big)\owedge g= W +\frac{1}{n-2} \operatorname{A}_{f} \owedge g.
\]
It turns out that $\operatorname{A}_f := \ricc_{f} -\frac{R_f}{2(n-1)}g$ (recall that $\ricc_f = \ricc + \nabla^2f$ and $R_f = R+\Delta f$).

\subsection*{The classes $\SF$ and $\SFf$}

A standard computation using Bianchi identities and the constancy of the scalar curvature shows that
  \[
  (M, g) \in \SF \quad  \Longleftrightarrow \quad \riem = \frac{\lambda}{2(n-1)} g \owedge g \quad \Longleftrightarrow \quad \begin{cases} W=0 \\ \ricc=\lambda g\end{cases}
  \]
In a similar fashion, using the constancy of $R_f$, we have

  \begin{equation}\label{equ-sff}
  (M, g, f) \in \SFf \quad  \Longleftrightarrow \quad \riem_f = \frac{\lambda}{2(n-1)} g \owedge g \quad \Longleftrightarrow \quad \begin{cases} W=0 \\ \ricc_f=\lambda g\end{cases}
  \end{equation}
  Note  that $\SF\subset \E$ and $\SFf \subset \Ef$; moreover, in dimension $n\geq 4$ every $f$-space form is a locally conformally flat gradient Ricci soliton (see Proposition \ref{pro-sff} and Section \ref{sec-rig} for more details).

  \subsection*{The classes $\LS$ and $\LSf$ (and also $\LSE$ and $\LSEf$)}

  One has
  \begin{equation*}
\nabla\operatorname{Riem} = \nabla W +\frac{1}{n-2} \nabla \big(\operatorname{A} \owedge g\big).
\end{equation*}
Moreover, $\nabla A=0$ implies the constancy of $R$, and is thus equivalent to $\nabla\ricc=0$. By orthogonality,
  \[
  (M, g) \in \LS \quad  \Longleftrightarrow \quad\nabla\riem =0 \quad \Longleftrightarrow \quad\begin{cases}\nabla W=0  \\ \nabla \ricc=0\end{cases}
  \]
  and analogously
  \begin{equation}\label{equ-lsf}
  (M, g, f) \in \LSf \quad  \Longleftrightarrow \quad\nabla\riem_f =0 \quad \Longleftrightarrow \quad\begin{cases}\nabla W=0  \\ \nabla \ricc_f=0\end{cases}
  \end{equation}
Note  that $\LS\subset \PR$ and $\LSf \subset \PRf$.
Moreover, since by definition $\LSE = \LS \cap \E$ and $\LSEf = \LSf \cap \Ef$, we get
   \[
  (M, g) \in \LSE \quad  \Longleftrightarrow \quad\begin{cases}\nabla\riem =0 \\ \ricc=\lambda g \end{cases}\quad \Longleftrightarrow \quad\begin{cases}\nabla W=0  \\ \ricc=\lambda g\end{cases}
  \]
  and analogously
  \begin{equation}\label{equ-lsef}
  (M, g, f) \in \LSEf \quad  \Longleftrightarrow \quad\begin{cases}\nabla\riem_f =0 \\ \ricc_f=\lambda g \end{cases}\quad \Longleftrightarrow \quad\begin{cases}\nabla W=0  \\ \ricc_f=\lambda g\end{cases}
  \end{equation}
  For the general discussion on the consequences of the previous equivalences, see again Section \ref{sec-rig}.


  \subsection*{The classes $\HC$ and $\HCf$}

 By Bianchi identities, $\diver({\riem})_{ijk} = R_{tijk, t} = R_{ik, j}-R_{ij, k}$; in particular, from the decomposition \eqref{Weyl}, on every Riemannian manifolds ($n\geq 3$) it holds
 \[
   \pa{\frac{n-3}{n-2}}R_{tijk, t} = W_{tijk, t} + \frac{(n-3)}{2(n-1)(n-2)}\pa{R_j\delta_{ik}-R_k\delta_{ij}}.
 \]
 This implies
  \begin{align*}
     (M,g)\in\HC \quad \Longleftrightarrow\quad  \ricc \text{ is a Codazzi tensor } \quad \Longleftrightarrow \quad \begin{cases}
        \diver{(W)} = 0 \\ \nabla R=0
      \end{cases}
  \end{align*}
Moreover, a simple computation shows also that
\begin{align*}
      (M,g)\in\HC \quad \Longleftrightarrow\quad \diver\sq{E\owedge g}=0
  \end{align*}
where $E:=\ricc-\frac{R}{2}g$ is the {\em Einstein tensor}, which has the property $\diver{(E)}=0$.

As far as $\HCf$ metrics are concerned, we have the
\begin{lemma}\label{lem-equiv}
 The following conditions are equivalent:
 \begin{itemize}
 \item[a)] $(M,g,f)\in\HCf$;

 \item[b)] The Bakry-Emery Ricci tensor, $\ricc_{f}$, is a Codazzi tensor.

 \item[c)] $(M,g,f)$ satisfies
 \begin{equation}\label{ICeq}
    \begin{cases}
      C_{ijk}+f_tW_{tijk} = D^{\nabla f}_{ijk} \\   R_{i} = 2 f_{t}R_{ti}
    \end{cases}
  \end{equation}

 \end{itemize}
where $D^{\nabla f}$ is the tensor defined in \eqref{def_D}.
\end{lemma}

\begin{proof}
  The equivalence $a) \Leftrightarrow b)$ follows from the commutation $f_{jkt}-f_{jtk}=f_{i}R_{ijkt}$ and
  \begin{align*}
  \pa{e^{-f}R_{ijkt}}_i &= e^{-f}\pa{R_{ijkt,i}-f_{i}R_{ijkt}} = e^{-f}\pa{R_{jt,k}-R_{jk,t}- f_{i}R_{ijkt}}\\ &= e^{-f}\big((\ricc_{f})_{jt,k}-(\ricc_{f})_{jk,t}\big) \,.
  \end{align*}
If $(M,g,f)\in\HCf$, we have
  \[
  R_{tijk, t} - f_tR_{tijk}=0,
  \]
  that is,
  \[
  R_{ij, k}-R_{ik, j} = -f_tR_{tijk}.
  \]
  Using in the previous relation the definition of the Cotton tensor $C$ and $D^{\nabla f}$, the decomposition of the Riemann curvature tensor \eqref{Weyl} and $R_{i} = 2 f_{t}R_{ti}$ we get the equivalence $a) \Leftrightarrow c)$.
\end{proof}
If $n\geq 4$, Lemma \ref{lem-equiv} and equation \eqref{def_Cotton_comp_Weyl} immediately imply
  \begin{align*}
      (M,g,f)\in\HCf \,\Longleftrightarrow\, \ricc_{f} \text{ is a Codazzi tensor } \,\Longleftrightarrow\, \begin{cases}
        W_{tijk,t} = \pa{\frac{n-3}{n-2}}\Big(f_{t}W_{tijk} - D^{\nabla f}_{ijk}\Big) \\ \nabla R=2\ricc(\nabla f,\cdot)
      \end{cases}
  \end{align*}
Let $E_{f}:=\ricc_{f}-\frac{R_{f}}{2} g$. In analogy with the classical case  we call it {\em $f$-Einstein tensor}. From the commutation rule $f_{ijk}-f_{ikj}=f_{t}R_{tijk}$ and from equation \eqref{eq-divkn}, we have
\begin{align*}
\diver(E_{f}\owedge g)_{ijk} =&\, \Big(f_{tjt}-\frac{1}{2}f_{ttj}\Big)\delta_{ik}-\Big(f_{tkt}-\frac{1}{2}f_{ttk}\Big)\delta_{ij} + \Big(R_{ik,j}-R_{ij,k}\Big)+f_{t}R_{tikj}\\
&-\frac{1}{2}\Big[(R_f)_{j}\delta_{ik}-(R_f)_{k}\delta_{ij}\Big]\\
=& \frac{1}{2}\big(R_{k}-2f_{t}R_{tk}\big)\delta_{ij}-\frac{1}{2}\big(R_{j}-2f_{t}R_{tj}\big)\delta_{ik}+\big(R_{ik,j}-R_{ij,k}\big)+f_{t}R_{tikj}
\end{align*}
Now, if $\diver\sq{E_{f}\owedge g}=0$, tracing the previous relation we obtain $\nabla R=2\ricc(\nabla f,\cdot)$. Hence
\begin{align*}
0=\big(R_{ik,j}-R_{ij,k}\big)+f_{t}R_{tikj} =e^{f}\diver(e^{-f}\riem)_{ijk}\,,
\end{align*}
i.e. $(M,g,f)\in\HCf$. Note that the converse is also true, and thus
 \begin{align}
      (M,g,f)\in\HCf \quad\Longleftrightarrow \quad  \diver\sq{E_{f}\owedge g}=0
  \end{align}
Moreover the latter equivalence enables us to define the non-gradient counterpart of $\HCf$, as we will see in Section \ref{sec-nongrad}.

 \subsection*{The classes $\Y$ and $\Yf$} Obviously, by Bianchi identities one has
  \begin{align*}
     (M,g)\in \Y \quad \Longleftrightarrow \quad \nabla R= 0  \quad \Longleftrightarrow \quad \diver(\ricc-R\,g)=0 \,.
        \end{align*}
As far as $\Yf$ metrics are concerned, since
$$
\big(E_{f}\owedge g\big)_{isks} = (n-2) \big(\ricc_{f}-R_{f} g \big)_{ik}\,,
$$
we have
  \begin{align*}
      (M,g,f)\in \Yf \quad \Longleftrightarrow \quad  \nabla R=2\ricc(\nabla f,\cdot)\quad \Longleftrightarrow \quad \diver\pa{\ricc_{f}-R_{f} g}=0
        \end{align*}
and, again, the latter equivalence enables us to define the non-gradient counterpart of $\Yf$ (see again Section \ref{sec-nongrad}).

\

\section{The rigid classes: $\SFf$, $\LSf$, $\LSEf$ and $\PRf$} \label{sec-rig}

\Small{

\begin{equation*}
\begin{array}{ccccccccc}
\,& \,& \LS & \subset & \PR & \, & \, & \, \\
\,& \rotatebox{45}{$\subset$}& \cup & \, & \cup &\rotatebox{45}{$\cap$} &\, & \,\\
\SF	&\subset &\LSE &\subset	&\E	&\subset	&\HC	&\subset	&\Y\\
\cap	&\,			& \cap	&\,			& \cap		&\,		&\cap  & \, &\cap\\
\mathbfcal{SF}_{f}	&\subset &\mathbfcal{LSE}_{f}	&\subset &\Ef	&\subset	&\HCf	&\subset	&\Yf\\
\,& \rotatebox{-45}{$\subset$}& \cap & \, & \cap &\rotatebox{-45}{$\cup$} &\, & \,\\
\,& \,& \mathbfcal{LS}_{f} & \subset & \mathbfcal{PR}_{f} & \, & \, & \, \\
\end{array}
\end{equation*}

}

\

\

\normalsize

First of all we observe that, as in the case of $\PR$, if $(M,g,f)\in\PRf$, i.e. $\nabla \ricc_{f}=0$ on $M$,
from the de Rham decomposition theorem, then $(M,g,f)$ is locally a Riemannian product of gradient Ricci solitons (see e.g. \cite[Sect. 16.12(i)]{besse} for a general splitting result concerning Codazzi tensor with constant eigenvalue).

\subsection*{$\SFf$: proof of Proposition \ref{pro-sff}}

Let $(M,g,f)\in\SFf$. First we observe that, in dimension $n=3$, $\SFf=\Ef$. Thus, from the classification of three dimensional gradient shrinking solitons, if $\lambda>0$, then $(M,g)$ is isometric, up to quotients, to either $\SS^{3}$ or $\RR\times\SS^{2}$ or $\RR^{3}$. On the other hand, if $n\geq 4$, from the conditions \eqref{equ-sff}, $(M,g,f)$ is a locally conformally flat gradient Ricci soliton. Proposition \ref{pro-sff} now follows from the classifications results in the shrinking (\cite{nw, zhang, pw}), steady (\cite{caoche1, catman}) and expanding (\cite{catman}) cases. To the best of our knowledge, the complete classification of locally conformally flat, gradient expanding Ricci solitons is still open; however it is known that around any regular point of $f$ the manifold $(M,g)$ is locally a warped product with codimension one fibers of constant sectional curvature.

\subsection*{$\LSf$ and $\LSEf$: proof of Proposition \ref{pro-lsf}}

Let $(M,g,f)\in\LSf$. As we have already observed, in dimension $n=3$, $\LSf=\PRf$. If $n\geq 4$, from equation \eqref{equ-lsf} we have that $(M,g,f)\in\PRf$ and the Weyl tensor is parallel, $\nabla W=0$. In particular, by a classical result of Roter (see \cite{derrot}), either $\nabla \riem =0$ or $W=0$. In the first case $(M,g,f)\in\LS$, while in the second case we are left with a locally conformally flat manifold with $\nabla \ricc_{f}=0$. Again, by de Rham decomposition theorem, we have just two possibilities:  $(M,g,f)\in\Ef$ with $W=0$ and thus, from equation \eqref{equ-sff}, $(M,g,f)\in\SFf$; $(M,g)$ splits as the product of two locally symmetric factors (a line with a space form or two space forms with opposite constant curvature and same dimension). In this latter case, $(M,g,f)\in\LS$.

Now let $(M,g,f)\in\LSEf$. In dimension $n=3$, $\LSEf=\Ef$, while if $n\geq 4$, by the previous discussion, either $(M,g,f)\in\SFf\cap\Ef=\SFf$, or $(M,g,f)\in\LS\cap\Ef$. In this case, in particular, the manifold is a gradient Ricci solitons which is also locally a product of Einstein metrics. Considering the universal cover and using classical results on concircular (gradient) vector fields (see e.g. \cite{tashiro}), we have that we can only have two type of factors in the decomposition: the Euclidean space or a (locally symmetric) Einstein manifold. This concludes the proof.

\

\section{The class $\HCf$: rigidity results, characterizations and examples} \label{sec-hcf}

\Small{

\begin{equation*}
\begin{array}{ccccccccc}
\,& \,& \LS & \subset & \PR & \, & \, & \, \\
\,& \rotatebox{45}{$\subset$}& \cup & \, & \cup &\rotatebox{45}{$\cap$} &\, & \,\\
\SF	&\subset &\LSE &\subset	&\E	&\subset	&\HC	&\subset	&\Y\\
\cap	&\,			& \cap	&\,			& \cap		&\,		&\cap  & \, &\cap\\
\SFf	&\subset &\LSEf	&\subset &\Ef	&\subset	&\mathbfcal{HC}_{f}	&\subset	&\Yf\\
\,& \rotatebox{-45}{$\subset$}& \cap & \, & \cap &\rotatebox{-45}{$\cup$} &\, & \,\\
\,& \,& \LSf & \subset & \PRf & \, & \, & \, \\
\end{array}
\end{equation*}

}

\

\

\normalsize

First of all, we recall that $(M, g, f)\in \HCf$ if and only if $\diver \big(e^{-f}\riem\big) =0$ or, equivalently, from Lemma \ref{lem-equiv}, if and only if \begin{equation*}
    \begin{cases}
      C_{ijk}+f_tW_{tijk} = D^{\nabla f}_{ijk} \\   R_{i} = 2 f_{t}R_{ti},
    \end{cases}
      \end{equation*}
      where
 \begin{align*}
   D^{\nabla f}_{ijk}=&\frac{1}{n-2}\pa{f_kR_{ij}-f_jR_{ik}}+\frac{1}{(n-1)(n-2)}f_t\pa{R_{tk}g_{ij}-R_{tj}g_{ik}}\\\nonumber
 &\,-\frac{R}{(n-1)(n-2)}\pa{f_k g_{ij}-f_j g_{ik}}.
 \end{align*}

\subsection*{Proof of Theorem \ref{teo-lcf}.} Let $(M,g,f)\in\HCf$; by the assumption of local conformal flatness, both the Cotton and the Weyl tensor vanish on $M$. From Lemma \ref{lem-equiv} we get that the tensor $D^{\nabla f}$ vanishes. Contracting with $\nabla f$ and using equation $R_{i} = 2 f_{t}R_{ti}$,  we obtain
\begin{align*}
0=&\,(n-1)(n-2)D_{ijk}f_{k}=(n-1)|\nabla f|^{2}R_{ij}-(n-1)R_{ik}f_{k}f_{j}\\\nonumber
&\,+R_{tk}f_{t}f_{k}g_{ij}-R_{tj}f_{t}f_{i}
 -|\nabla f|^{2}R\,g_{ij}+Rf_{i}f_j\\
 =&\, (n-1)|\nabla f|^{2}R_{ij}-|\nabla f|^{2}R\,g_{ij}-\frac{n-1}{2}R_{i}f_{j}-\frac{1}{2}f_{i}R_{j}+\frac{1}{2}\langle \nabla R,\nabla f\rangle g_{ij} +R f_{i}f_{j} \,.
\end{align*}
By symmetry, we get $R_{i}f_{j}=R_{j}f_{i}$, i.e. $dR \wedge df =0$. In particular, $\nabla f$ is an eigenvector of the Ricci tensor and, from $2\ricc(\nabla f, \nabla f)=\langle \nabla R,\nabla f\rangle$ we obtain
$$
0 = (n-1)|\nabla f|^{2}R_{ij}-|\nabla f|^{2}R\,g_{ij}-\frac{n}{2}R_{i}f_{j}+\ricc(\nabla f, \nabla f) g_{ij} +R f_{i}f_{j}\,.
$$
Now, around a regular point of $f$, pick any orthonormal frame $e_{1},\ldots,e_{n}$ which diagonalize the Ricci tensor. Since $\nabla f$ is an eigenvector of Ricci, without loss of generality we can set $e_{1} =\frac{\nabla f}{|\nabla f|}$. Denote by $\mu_{k}$, $k=1,\ldots,n$ the corresponding eigenvalues. Then, for every $k\geq 2$, we have
$$
0 =  |\nabla f|^{2}\big((n-1)\mu_{k}-R+\mu_{1}\big)\,.
$$
Thus, around a regular point of $f$, one has $\mu_{k}=\frac{1}{n-1}(R-\mu_{1})$ for every $k\geq 2$. In particular, around a regular point of $f$, either the Ricci is proportional to the metric or it has an eigenvalue of multiplicity $(n-1)$ and another of multiplicity $1$.

Now suppose that $f$ is not constant. We have shown that either the metric is locally Einstein (thus of constant curvature), or the Ricci tensor has two eigenvalues of multiplicity 1 and $(n-1)$. In the first case, the manifold must be locally isometric to a space form. In the second case, since the Cotton tensor $C$ vanishes, the Schouten tensor $\ricc-\frac{1}{2(n-1)}R\,g$ is a Codazzi tensor with at most two distinct eigenvalues of multiplicity 1 and $(n-1)$. Hence, by general results on Codazzi tensors with this property (see \cite{mer, besse, catmanmazcod}) we get that the manifold $(M,g)$ is locally a warped product with codimension one fibers. Since the manifold is locally conformally flat, the fibers must have constant sectional curvature.

This concludes the proof of Theorem \ref{teo-lcf}.

\subsection*{Proof of Theorem \ref{teo-4d} and Corollary \ref{cor-hcf}.} First of all we recall the decomposition of the bundle two forms $\Lambda^{2}$ in dimension four
\begin{equation}\label{dec}
\Lambda^{2}=\Lambda^{+} \oplus \Lambda^{-}\,.
\end{equation}
These subbundles are by definition the eigenspaces of the Hodge operator
$$
\star:\Lambda^{2}\rightarrow \Lambda^{2}
$$
corresponding respectively to the eigenvalue $\pm 1$. In the literature, sections of $\Lambda^{+}$ are called {\em self-dual} two-forms, whereas sections of $\Lambda^{-}$ are called {\em anti-self-dual} two-forms. Now, since the curvature tensor $\riem$ may be viewed as a map $\mathcal{R}:\Lambda^{2}\ra\Lambda^{2}$, according to \eqref{dec} we have the curvature decomposition
\begin{displaymath}
\mathcal{R}=\left(\begin{array}{c|c}
W^{+}+\frac{R}{12}\,I & \overset{\circ}{\ricc} \\
\hline
\overset{\circ}{\ricc} & W^{-}+\frac{R}{12}\,I \end{array}\right), \end{displaymath}
where
$$
W = W^{+} + W^{-}
$$
and the self-dual and anti-self-dual $W^{\pm}$ are trace-free endomorphisms of $\Lambda^{\pm}$, $I$ is the identity map of $\Lambda^{2}$ and $\overset{\circ}{\ricc}$ represents the trace-free Ricci curvature $\ricc-\frac{R}{4}g$. Recall the Hirzebruch signature formula (see e.g. \cite{besse})
$$
48\pi^{2} \tau(M)= \int_{M} |W^{+}|^{2} - \int_{M}|W^{-}|^{2} \,.
$$
Assume that $\tau(M)\neq 0$ and let $(M,g,f)\in\HCf$, for some potential function $f$; assume also that, in harmonic coordinates, $g$ and $f$ are real analytic. From Lemma \ref{lem-equiv}, the Bakry-Emery Ricci tensor $\ricc_{f}$ is Codazzi. In particular the following property holds:
\begin{lemma}\label{lem-bou} Let $T$ be a Codazzi tensor on a four dimensional Riemannian manifold $(M,g)$. Then, at any point $x$ where $T$ is not a multiple of $g$, the endomorphisms $W^{+}$ of $\Lambda^{+}$ and $W^{-}$ of $\Lambda^{-}$ have equal spectra at $x$.
\end{lemma}
This result was proved by Bourguignon \cite{bou2tr} (see also \cite{dershe}) and used in the context of manifolds with harmonic curvature. By analyticity, it implies that either $\ricc_{f}$ is proportional to the metric  (i.e. $(M,g,f)\in\Ef$), or $W^{+}$ and $W^{-}$ have equal spectra on $M$. But this contradicts the topological assumption on $\tau(M)$ and the first part of Theorem \ref{teo-4d} is proved.

Assume now that $(M,g,f)\in \HCfl$, without imposing extra regularity on $g$ ad $f$. We have that
\begin{equation}\label{WHCl}
\diver \big(e^{-f}\riem\big) =0\, \quad\hbox{and}\,\quad R+\Delta f = n \lambda \,.
\end{equation}
From Lemma \ref{lem-equiv}, the Bakry-Emery Ricci tensor $\ricc_{f}$ is a Codazzi tensor with constant trace. Equivalently,
$$
\overset{\circ}{\ricc_{f}}:=\ricc_{f} - \frac{R_{f}}{n}g
$$
is a trace-free Codazzi tensor. In particular, we have the following regularity lemma which follows from a general results of Kazdan \cite{kazdan} (see also \cite{gursky, catconf} for some applications).
\begin{lemma}\label{l-kaz}
Let $\overset{\circ}{T}$ be a, non-trivial, trace-free Codazzi tensor on a Riemannian manifold $(M,g)$ and let $\Omega_{0}=\{\,x\in M^{n}:\, |\overset{\circ}{T}|(x)\neq 0\,\}$. Then $\hbox{Vol}\,(M\setminus\Omega_{0}) = 0$.
\end{lemma}
Using this, together with Lemma \ref{lem-bou}, one has that either $\overset{\circ}{\ricc_{f}}\equiv 0$ (i.e. $(M,g,f)\in\Ef$), or
$$
\int_{M}|W^{+}|^{2} = \int_{M}|W^{-}|^{2} \,,
$$
which again contradicts the assumption $\tau(M)\neq 0$, and the second part of Theorem \ref{teo-4d} is proved.

Finally, Corollary \ref{cor-hcf} follows immediately from Theorem \ref{teo-4d} ii) and the classification of half conformally flat gradient Ricci solitons in \cite{chenwang}.

\subsection*{Proof of Proposition \ref{pro-sec} and Corollary \ref{cor-tac}.}

Let $(M,g,f)\in \HCfl$. Then $\overset{\circ}{\ricc_{f}}$ is a trace-free Codazzi tensor. In particular (see \cite{besse} or \cite{catconf}), the following Weitzenb\"ock formula holds
\begin{equation}\label{eq-wei}
\frac{1}{2}\Delta |\overset{\circ}{\ricc_{f}}|^{2} \,=\, |\nabla \overset{\circ}{\ricc_{f}}|^{2} - R_{ikjl} (\overset{\circ}{\ricc_{f}})_{ij}(\overset{\circ}{\ricc_{f}})_{kl} + R_{jk} (\overset{\circ}{\ricc_{f}})_{ij}(\overset{\circ}{\ricc_{f}})_{ik} \,.
\end{equation}
Let $\{e_{i}\}$, $i=1,\ldots, n$, be the set of the eigenvectors of $\overset{\circ}{\ricc_{f}}$ and let $\mu_{i}$ be the corresponding eigenvalues. Moreover, let $k_{ij}$ be the sectional curvature defined by the two-plane spanned by $e_{i}$ and $e_{j}$. One has
$$
-R_{ikjl} (\overset{\circ}{\ricc_{f}})_{ij}(\overset{\circ}{\ricc_{f}})_{kl} + R_{jk} (\overset{\circ}{\ricc_{f}})_{ij}(\overset{\circ}{\ricc_{f}})_{ik} = - \sum_{i,j=1}^{n} \mu_{i}\mu_{j} k_{ij} + \sum_{i,j=1}^{n} \mu_{i}^{2} k_{ij} = \sum_{i<j} (\mu_{i}-\mu_{j})^{2}k_{ij} \geq 0\,,
$$
since $k_{ij} > 0$ for all $i,j=1,\ldots,n$. Using this and integrating the Weitzenb\"ock formula, we get that $\overset{\circ}{\ricc_{f}}$ has to be zero on $M$, i.e. $(M,g,f)\in\Ef$. This proves Proposition \ref{pro-sec}.

Corollary \ref{cor-tac} simply follows from Proposition \ref{pro-sec} and the classification of compact gradient Ricci solitons with positive curvature operator (see \cite{bohwil}).

\subsection*{Two examples}

We construct two examples of Riemannian manifolds in $\HCf$, following the construction for the harmonic curvature case given by Derdzinski in \cite{derd3}, following the same notation to highlight the similarities. Let $I\subseteq \RR$ be an interval, $F\in C^{\infty}(I)$ a smooth positive function on $I$ and $(N,h)$ an $(n-1)$-dimensional Einstein manifold with constant scalar curvature $k$. We consider the warped product manifold $\Big(M=I\times N, g=dt^{2}+F(t)h\Big)$. Letting the indices $i,j,k$ run through $1,\ldots,n-1$ and given a local chart $t=x^{0},x^{1},\ldots,x^{n-1}$ for $I\times N$, we have $g_{00}=1, g_{0i}=0$, $g_{ij}=F\,h_{ij}$ and the components of the Ricci tensor $\ricc$ and its covariant derivative $\nabla\ricc$ are given by
\begin{align}\label{ricex}
R_{00}&=-\frac{n-1}{4}\Big(2q''+(q')^{2}\Big)\,,\quad R_{0i}=0 \,,\\\nonumber
R_{ij}&= \left(\frac{k}{n-1}-\frac{1}{4}e^{q}\Big(2q''+(n-1)(q')^{2}\Big)\right) h_{ij} \,,\\\nonumber
\nabla_{0}R_{00} &= -\frac{n-1}{2}\Big(q'''+q'q''\Big)\,, \quad \nabla_{0}R_{i0}=\nabla_{i}R_{00}=0 \,,\\\nonumber
\nabla_{0}R_{ij}&= -\left(\frac{k}{n-1}q'+\frac{1}{2}e^{q}q'''+\frac{n-1}{2}e^{q}q'q''\right)h_{ij}\,,\\
\nabla_{i}R_{0j} &= -\left(\frac{k}{2(n-1)}+\frac{n-2}{4}e^{q}q'q''\right)h_{ij}\,,\quad \nabla_{k}R_{ij}=0 \,,
\end{align}
where $q=\log F$. Since $\nabla_{0}R_{i0}=\nabla_{i}R_{00}=R_{pi00}f_{p}=0$, the condition $\diver(e^{-f}\riem)=0$ is equivalent to
\begin{equation}\label{eq127}
\nabla_{0}R_{ij}-\nabla_{i}R_{0j}+R_{0ji0}\nabla_{0}f=0 \,.
\end{equation}
Using the expression of the Riemann curvature tensor in terms of the Christoffel symbols and the fact  that $\Gamma^{j}_{0i}=\frac{1}{2}q'h_{ij}$, $\Gamma^{i}_{00}=\Gamma^{0}_{i0}=\Gamma^{0}_{00}=0$, one has
\begin{align*}
R_{0ij0} &= \partial_{0}\Gamma^{j}_{i0}-\partial_{i}\Gamma^{j}_{00}+\Gamma^{p}_{i0}\Gamma^{j}_{0p}-\Gamma^{p}_{00}\Gamma^{j}_{ip} \\
&=\partial_{0}\Big(\frac{1}{2}q'h_{ij}\Big)+\Gamma^{k}_{i0}\Gamma^{j}_{0k}\\
&=\frac{1}{4}\Big(2q''+(q')^{2}\Big)h_{ij} \,.
\end{align*}
Hence, equation \eqref{eq127} is equivalent to the following differential equation for the function $q$
\begin{equation}\label{eq145}
q'''+\frac{n}{2}q'q''+\frac{k}{n-1}e^{-q}q'=\frac{1}{2}\Big(2q''+(q')^{2}\Big)f' \,.
\end{equation}
First of all, a simple computations shows that the choice  $k=0$,
$$
q(t)=t^{2}\,\quad f(t)=\frac{n}{2}\log(1+t^{2})
$$
gives a solution to the equation. Hence we have that, given any $(n-1)$-dimensional Riemannian Ricci flat manifold $(N,h)$, one has
$$
\Big(M=\RR\times N,\,g=dt^{2}+e^{t^{2}}h,\,f(t)=\frac{n}{2}\log(1+t^{2})\Big) \in \HCf \,.
$$
Now we want to construct a compact example. Integrating equation \eqref{eq145}, we get
\begin{equation}\label{eq128}
q''+\frac{n}{4}(q')^{2}-\frac{k}{n-1}e^{-q}= \frac{1}{2}\int\Big(2q''+(q')^{2}\Big)f' \,.
\end{equation}
Now, we suppose that, given a function $q$ defined on some interval $I$, we can find $f$ solving
\begin{equation}\label{eq129}
\frac{1}{2}\Big(2q''+(q')^{2}\Big)f' = \frac{\varepsilon k}{n-1}q'e^{-q} \,,
\end{equation}
for some $\varepsilon>0$. Plugging this into \eqref{eq128}, we reduce problem in solving
\begin{equation}\label{eq130}
q''+\frac{n}{4}(q')^{2}-\frac{k-\varepsilon}{n-1}e^{-q}=\frac{4}{n}C\,,
\end{equation}
for some constant $C\in\RR$. Letting $\varphi:=e^{\frac{n}{4}q}$, we obtain the ODE
\begin{equation}\label{eq131}
\varphi''-\frac{n(k-\varepsilon)}{4(n-1)}\varphi^{1-\frac{4}{n}}=C\varphi \,,
\end{equation}
for some constant $C\in\RR$. It was shown in \cite[Theorem 1]{derd3} that, if $k>\varepsilon$ and $C<0$, this equation have non-constant positive periodic smooth solutions, defined in $\RR$. Now, let $\varphi=e^{\frac{n}{4}q}$ be a solution, then from the equation \eqref{eq130}, one has
$$
2q''+(q')^{2}=\frac{8}{n}C-\frac{n-2}{2}(q')^{2}+\frac{2(k-\varepsilon)}{n-1}e^{-q}\leq \frac{8}{n}\left(C+\frac{2(k-\varepsilon)}{n-1}\right)<0 \,
$$
provided $C<-\frac{2(k-\varepsilon)}{n-1}$. Then, under this assumption, we can always integrate equation \eqref{eq129} and find the potential function $f$.

Now, let $(N,h)$ be a compact $(n-1)$-dimensional Einstein manifold with (constant) positive scalar curvature $k>\varepsilon>0$; choose a non-constant, positive, periodic function $F$ on $\RR$ such that $\varphi=F^{\frac{n}{4}}$ satisfies \eqref{eq131} for some constant $C<-\frac{2(k-\varepsilon)}{n-1}$; and choose $f=f(t)$ solving equation \eqref{eq129}. Then, following the precise construction in \cite[Section 3]{derd3}, we can define a compact Riemannian quotient of $\Big(\RR\times N, g=dt^{2}+F(t)h, f(t)\Big)$, $(\widetilde{M}, \widetilde{g}, \widetilde{f})$, such that $\widetilde{M}$ is diffeomorphic to $\SS^{1}\times N$ and $\widetilde{g}$ has weighted harmonic curvature, namely $(\widetilde{M}, \widetilde{g}, \widetilde{f})\in\HCf$.

\

\section{The class $\Yf$: a possible generalization of the Yamabe problem, obstructions and examples}\label{sec-yf}

\Small{

\begin{equation*}
\begin{array}{ccccccccc}
\,& \,& \LS & \subset & \PR & \, & \, & \, \\
\,& \rotatebox{45}{$\subset$}& \cup & \, & \cup &\rotatebox{45}{$\cap$} &\, & \,\\
\SF	&\subset &\LSE &\subset	&\E	&\subset	&\HC	&\subset	&\Y\\
\cap	&\,			& \cap	&\,			& \cap		&\,		&\cap  & \, &\cap\\
\SFf	&\subset &\LSEf	&\subset &\Ef	&\subset	&\HCf	&\subset	&\mathbfcal{Y}_{f}\\
\,& \rotatebox{-45}{$\subset$}& \cap & \, & \cap &\rotatebox{-45}{$\cup$} &\, & \,\\
\,& \,& \LSf & \subset & \PRf & \, & \, & \, \\
\end{array}
\end{equation*}

}

\

\

\normalsize

In this section we consider the class of Riemannian manifolds $(M,g,f) \in \Yf$, i.e. satisfying the condition
\begin{equation}\label{yf}
\nabla R = 2\ricc(\nabla f, \cdot)\,.
\end{equation}
This equation is a meaningful generalization of the one for Yamabe metrics ($\Y$) and can be seen as a very special prescription on the gradient of the scalar curvature, connecting the Ricci tensor with its trace {\em via} the potential function.

From this point of view, it is natural to study the following problems on a given manifold $M$:

\begin{itemize}

\item[(A)] having fixed $f\in C^{\infty}(M)$, there exists a metric $g$ such that $(M,g,f)\in \Yf$?

\item[(B)] having fixed $f\in C^{\infty}(M)$ and a metric $g_{0}$, there exists a conformal metric $g\in[g_{0}]$ such that $(M,g,f)\in \Yf$?

\end{itemize}

More generally, one could ask the question

\begin{itemize}

\item[(C)] there exist a metric $g$ and a smooth function $f\in C^{\infty}(M)$ such that $(M,g,f)\in \Yf$?

\end{itemize}
Clearly the answer to (C) is positive, since it is always possible to construct a (complete) metric with constant (negative) scalar curvature (\cite{aubin} and \cite{blakal}). Furthermore, when $f$ is constant, (B) boils down to the well known Yamabe problem, which is completely solved when $M$ is compact (see e.g.  \cite{leepar}).

In the same spirit of the work of Kazdan and Warner (see \cite{kazwar1}), here we prove some obstructions to problem (B), reserving to subsequent works a thorough study of problems (A) and (B) in the case $f$  nonconstant.

First of all we recall that a smooth vector field $X$  is a {\em conformal vector field} on $(M,g)$ if and only if
\begin{equation}\label{eq-cvf}
\mathcal{L}_{X} g = \frac{2\diver{(X)}}{n} g \,,
\end{equation}
where $\mathcal{L}_{X}g$ denotes the Lie derivative of the metric in the direction $X$. Equation \eqref{yf}, together with the the well known Kazdan-Warner identity (see \cite{kazwar1, bouezi}), gives the following integral condition for compact $f$-Yamabe metrics. For the sake of completeness, we include a simple proof.
\begin{lemma}\label{lem-kaz} If $M$ is compact and $(M,g,f)\in\Yf$, then, for every conformal vector field $X$ on $(M,g)$, one has
$$
\int_{M} \ricc(\nabla f, X)\, dV = 0 \,.
$$
\end{lemma}
\begin{proof} From equation \eqref{yf} and the fact that $X$ satisfies
$$
X_{ij}+X_{ji} = \frac{2\diver{(X)}}{n} g_{ij} \,,
$$
one has
\begin{align*}
2\int_{M} \ricc(\nabla f, X)\, dV &= \int_{M} \langle \nabla R, X \rangle\,dV = \frac{2n}{n-2} \int_{M} \rdc_{ij,j} X_{i}  \,dV \\
&= -\frac{n}{n-2} \int_{M} \rdc_{ij} \Big(X_{ij}+X_{ji}\Big)\,dV = 0 \,,
\end{align*}
where we have used integration by parts and Bianchi identity for the trace-less Ricci tensor $\rd$, i.e. in coordinates $\rdc_{ij}=R_{ij}-\frac{R}{n}\delta_{ij}$.
\end{proof}

When $(M,g_{0})$ supports a nontrivial (nonvanishing) conformal gradient vector field, the previous lemma gives an obstruction to the existence of a $f$-Yamabe metric in the conformal class $[g_{0}]$.

\begin{cor} \label{cor-kazsphere} Let $(M,g_{0})$ be a compact Riemannian manifold and $X=\nabla f$, $f\in C^{\infty}(M)$, be a nontrivial conformal gradient vector field on $(M,g_{0})$. Then, there are no conformal metrics $g\in[g_{0}]$ such that $(M,g,f)\in\Yf$.
\end{cor}
\begin{proof} Let $g\in[g_{0}]$. By the conformal invariance of equation \eqref{eq-cvf}, we have that $X=\nabla f$ is also a conformal vector field for $(M,g)$, i.e. the potential function $f$ satisfies
$$
\nabla^{2}f = \frac{\Delta f}{n} g \,,
$$
where all the covariant derivatives refer to the metric $g$. Integrating Bochner formula
$$
\frac{1}{2}\Delta|\nabla f|^{2} = |\nabla^{2}f|^{2}+\ricc(\nabla f, \nabla f) + \langle \nabla f, \nabla \Delta f \rangle
$$
on $M$, one obtain
$$
\int_{M} \ricc(\nabla f,\nabla f)\,dV = \int_{M}|\Delta f|^{2}\,dV - \int_{M}|\nabla^{2}f|^{2}\,dV = \frac{n-1}{n}\int_{M}|\Delta f|^{2}\,dV \,.
$$
Suppose now that $(M,g,f)\in\Yf$. Then, using Lemma \ref{lem-kaz} with $X=\nabla f$, we obtain $\Delta f=0$, i.e. $f$ is constant on $M$,  which is a contradiction.
\end{proof}

In particular, from this result Proposition \ref{pro-sphere} in Section \ref{sec-main}, namely we have the following:

\begin{proposition} If $f \in C^{\infty}(\SS^{n})$ is a first  spherical harmonic on the round sphere $(\SS^{n},g_0)$, then there are no conformal metrics $g\in[g_{0}]$ such that $(M,g,f)\in\Yf$.
\end{proposition}

Note that, by a classical result of Tashiro \cite{tashiro}, every compact manifold supporting a nontrivial (nonvanishing) conformal gradient vector field is conformal to the round sphere $\SS^{n}$.

\subsection*{An example}

Let $I\subseteq \RR$ be an interval, $F\in C^{\infty}(I)$ a smooth positive function on $I$ and $(N,h)$ an $(n-1)$-dimensional manifold with Ricci curvature $\rho$. As in Section \ref{sec-hcf}, we consider the warped product manifold $\Big(M=I\times N, g=dt^{2}+F(t)h\Big)$. Letting the indices $i,j,k$ run through $1,\ldots,n-1$ and given a local chart $t=x^{0},x^{1},\ldots,x^{n-1}$ for $I\times N$, we have $g_{00}=1, g_{0i}=0$, $g_{ij}=F\,h_{ij}$ and the components of the Ricci tensor $\ricc$ are given by
\begin{align}\label{ricex}
R_{00}&=-\frac{n-1}{4}\Big(2q''+(q')^{2}\Big)\,,\quad R_{0i}=0 \,,\\\nonumber
R_{ij}&= \rho_{ij}-\frac{1}{4}e^{q}\Big(2q''+(n-1)(q')^{2}\Big) h_{ij} \,.
\end{align}
where $q=\log F$. Suppose that $(N,h)$ has constant scalar curvature $k$. Then, the scalar curvature of $(M,g)$ is given by
$$
R = -\frac{n-1}{4}\Big(4q''+n(q')^{2}\Big) + k e^{-q}\,.
$$
On the other hand, if the potential function $f$ is radial, then
$$
\ricc(\nabla f) = g^{00}R_{00}f' = -\frac{n-1}{4}\Big(2q''+(q')^{2}\Big)f' \,.
$$
Thus, equation \eqref{yf} is equivalent to the following ODE
$$
q'''+\frac{n}{2}q'q''+\frac{k}{n-1}e^{-q}q'=\frac{1}{2}\Big(2q''+(q')^{2}\Big)f' \,.
$$
Notice that this equation coincide with \eqref{eq145}. Hence, again the choice $k=0$ and
$$
q(t)=t^{2}\,\quad f(t)=\frac{n}{2}\log(1+t^{2})
$$
gives a solution to the equation. In this case we have that, given any $(n-1)$-dimensional Riemannian scalar flat manifold $(N,h)$, one has
$$
\Big(M=\RR\times N,\,g=dt^{2}+e^{t^{2}}h,\,f(t)=\frac{n}{2}\log(1+t^{2})\Big) \in \Yf \,.
$$
Moreover, if $(N,h)$ is not Ricci flat, it is easy to see that $(M,g,f)\notin\HCf$.

On the other hand, following the construction in Section \ref{sec-hcf}, given any compact $(n-1)$-dimensional manifold $(N,h)$ with constant positive scalar curvature $k>0$, we can construct a $f$-Yamabe metric on a compact manifold $M$ diffeomorphic to $\mathbb{S}^{1}\times N$. As before, if $(N,h)$ is not Einstein, then this solution $(M,g,f)\notin\HCf$.

\

\section{Nongradient canonical metrics} \label{sec-nongrad}

We provide here the complete generalization of the framework constructed in the previous sections to the {\em nongradient} setting. Again, the starting of our analysis are Ricci solitons, namely Riemannian manifolds $(M,g)$ for which there exists a vector field $X\in\mathfrak{X}(M)$ such that
$$
\ricc_{X}:=\ricc + \frac{1}{2}\mathcal{L}_{X}g = \lambda g
$$
for some constant $\lambda\in\RR$, where $\mathcal{L}_{X}g$ denotes the Lie derivative of the metric in the direction $X$. In this we case we say that $(M,g,X)\in \EX$. In this section we use the following notation: $E_{X}:=\ricc_{X}-\frac{R_{X}}{2}g$, $R_{X}:=R+\diver(X)$ and $\mathcal{A}^{X}$  the {\em antisymmetric part} of the $\nabla X$, i.e., in local coordinates, $\mathcal{A}^{X}_{ij}=X_{ij}-X_{ji}$, in such a way that $\nabla X = \frac{1}{2}\big(\mathcal{A}^{X}+\mathcal{L}_{X}g\big)$. If $X = \nabla f$ for some smooth potential function $f$, then the soliton is a gradient Ricci soliton ($\Ef$); note that, in this case, $\mathcal{A}^{X}=0$ and $\frac{1}{2}\mathcal{L}_{X}g= \nabla^{2}f$.

It follows from the work of Perelman \cite{perelman} (see \cite{emilanman} for a direct proof) that any compact Ricci soliton is actually a gradient Ricci soliton. In particular it is well known that, if $\lambda\leq 0$, then $(M,g,X)\in \E$. Moreover, Naber \cite{naber} has shown that any shrinking ($\lambda >0$) Ricci soliton with bounded curvature has a gradient soliton structure. On the other hand, steady ($\lambda =0$) and expanding ($\lambda<0$) Ricci solitons which do not support a gradient structure were found in \cite{lau, lot, baidan, bai}.

In order to introduce the nongradient counterparts of the $f$-canonical metrics that we have introduced in Definition \ref{ladefinizione}, we note that we have defined the classes $\HCf$ and $\Yf$ imposing the vanishing of the divergence of the ``weighted'' tensors $e^{-f}\riem$ and $e^{-f}\ricc$. Fortunately, we have shown in Section \ref{sec-frame} that these structures can be characterized using the tensor $\ricc_{f}$: this allows us to give the following

\begin{defi}\label{ladefinizioneX} Let $(M,g)$ be a $n$-dimensional, $n\geq 3$, Riemannian manifold with metric $g$. We say that the triple $(M,g,X)$ belongs to
\begin{itemize}

\item $\SFX$ ({\em $X$-space forms}) if there exist $X\in\mathfrak{X}(M)$ and $\lambda\in\RR$ such that
$$
\riem_{X} := \riem + \frac{1}{n-2}\Big(\frac{1}{2}\mathcal{L}_{X}g-\frac{\diver(X)}{2(n-1)}\,g\Big)\owedge g = \frac{\lambda}{2(n-1)} g \owedge g\,;
$$

\item $\LSEX$ ({\em $X$-locally symmetric Einstein metrics}) if there exist $X\in\mathfrak{X}(M)$ and $\lambda\in\RR$ such that
$$
\nabla \big(\riem_{X}\big) =0 \quad\hbox{and}\quad \ricc_{X} = \lambda g\,;
$$

\item $\EX$ ({\em Ricci solitons}) if there exist $X\in\mathfrak{X}(M)$ and $\lambda\in\RR$ such that
$$
\ricc_{X} = \lambda\,g \,;
$$

\item $\HCX$ ({\em $X$-harmonic curvature metrics}) if there exist $X\in\mathfrak{X}(M)$  such that
\begin{align*}
      \diver\sq{E_{X}\owedge g}=0  \,.
  \end{align*}
\item $\YX$ ({\em $X$-Yamabe metrics}) if there exist $X\in\mathfrak{X}(M)$  such that
$$
\nabla R = 2 \ricc(X,\cdot) + \diver(\mathcal{A}^{X})
$$
where $\diver(\mathcal{A}^{X})_{i}=\mathcal{A}^{X}_{ij,j}=X_{ij,j}-X_{ji,j}$.
\end{itemize}
Moreover, we say that $(M,g,f)$ belongs to
\begin{itemize}

\item $\LSX$ ({\em $X$-locally symmetric metrics}) if there exist $X\in\mathfrak{X}(M)$ such that
$$
\nabla \big(\riem_{X}\big) =0 \,;
$$

\item $\PRX$ ({\em metrics with parallel $X$-Ricci tensor}) if there exist $X\in\mathfrak{X}(M)$ such that
$$
\nabla \big(\ricc_{X}\big) =0  \,.
$$
\end{itemize}
\end{defi}

Note that, when $X=\nabla f$, we recover the corresponding sets in \eqref{cattedrale}; in this latter case, we say that the structure is {\em  gradient}. In particular, we have

\begin{equation*}
\begin{array}{ccccccccc}
\,& \,& \LSf & \subset & \PRf & \, & \, & \, \\
\,& \rotatebox{45}{$\subset$}& \cup & \, & \cup &\rotatebox{45}{$\cap$} &\, & \,\\
\SFf	&\subset &\LSEf &\subset	&\Ef	&\subset	&\HCf	&\subset	&\Yf\\
\cap	&\,			& \cap	&\,			& \cap		&\,		&\cap  & \, &\cap\\
\SFX	&\subset &\LSEX	&\subset &\EX	&\subset	&\HCX	&\subset	&\YX\\
\,& \rotatebox{-45}{$\subset$}& \cap & \, & \cap &\rotatebox{-45}{$\cup$} &\, & \,\\
\,& \,& \LSX & \subset & \PRX & \, & \, & \, \\
\end{array}
\end{equation*}

\subsection*{The class $\SFX$}

Using the constancy of $R_{X}=R+\diver(X)$, which follows tracing twice the definition equation, we have
  \begin{equation*}
  (M, g, f) \in \SFX \quad  \Longleftrightarrow \quad \riem_X = \frac{\lambda}{2(n-1)} g \owedge g \quad \Longleftrightarrow \quad \begin{cases} W=0 \\ \ricc_X=\lambda g\end{cases}
  \end{equation*}
  Note  that $\SFX \subset \EX$; moreover, in dimension $n\geq 4$ every $X$-space form is a locally conformally flat Ricci soliton. In particular, using the results in \cite{catmanmazpde}, the analogue of Proposition \ref{pro-sff} holds.

\subsection*{The classes $\LSX$ and $\LSEX$}
One has
  \begin{equation*}
\nabla\operatorname{Riem_{X}} = \nabla W +\frac{1}{n-2} \nabla \big(\operatorname{A_{X}} \owedge g\big)
\end{equation*}
where $A_{X}:=\ricc_{X}-\frac{R_{X}}{2(n-1)}g$. Moreover, $\nabla A_{X}=0$ implies the constancy of $R_{X}$, and is thus equivalent to $\nabla\ricc_{X}=0$. By orthogonality,

  \begin{equation*}
  (M, g, f) \in \LSX \quad  \Longleftrightarrow \quad\nabla\riem_X =0 \quad \Longleftrightarrow \quad\begin{cases}\nabla W=0  \\ \nabla \ricc_X=0\end{cases}
  \end{equation*}
and, obiouvsly,
  \begin{equation*}
  (M, g, f) \in \LSEX \quad  \Longleftrightarrow \quad\begin{cases}\nabla\riem_X =0 \\ \ricc_X=\lambda g \end{cases}\quad \Longleftrightarrow \quad\begin{cases}\nabla W=0  \\ \ricc_X=\lambda g\end{cases}
  \end{equation*}
Even in this more general situation, the analogue of Proposition \ref{pro-lsf} holds. Note that, for the $\LSEX$, one has to use general results for homothetic vector fields contained, for instance, in \cite{tashiro}.

  \subsection*{The class $\HCX$}

By definition
 \begin{align*}
      (M,g,X)\in\HCX \quad\Longleftrightarrow \quad  \diver\sq{E_{X}\owedge g}=0\,,
  \end{align*}
where $E_{X}=\ricc_{X}-\frac{R_{X}}{2}g$ and $R_{X}=R+\diver(X)$.
We claim that
  \begin{align*}
      (M,g,X)\in\HCX \,\Longleftrightarrow\, \ricc_{X} \text{ is a Codazzi tensor } \,\Longleftrightarrow\, \begin{cases}
         W_{tijk,t} = \pa{\frac{n-3}{n-2}}\Big(X_{t}W_{tijk} - D^{X}_{ijk}\Big) \\ \nabla R = 2 \ricc(X,\cdot) + \diver(\mathcal{A}^{X})
      \end{cases}
  \end{align*}
where
  \begin{align*}
  D^X_{ijk} &= \frac{1}{n-2}\pa{X_kR_{ij}-X_jR_{ik}}+\frac{1}{(n-1)(n-2)}\pa{X_tR_{tk}\delta_{ij}-X_tR_{tj}\delta_{ik}}\\ \nonumber
 & -\frac{R}{(n-1)(n-2)}\pa{X_k\delta_{ij}-X_j\delta_{ik}}\\ \nonumber &+\frac 12\pa{X_{kji}-X_{jki}}+\frac{1}{2(n-1)}\sq{\pa{X_{tkt}-X_{ktt}}\delta_{ij}-\pa{X_{tjt}-X_{jtt}}\delta_{ik}}.
\end{align*}
This definition follows from a previous work of the authors \cite{catmasmonrig}, where we derived the so called integrability conditions for nongradient Ricci solitons.

Assume $\diver\sq{E_{X}\owedge g}=0$. From equation \eqref{eq-divkn} one has
$$
(E_{X})_{tj,t} \delta_{ik}-(E_{X})_{tk,t} \delta_{ij} = (E_{X})_{ij,k} -(E_{X})_{ik,j}\,.
$$
Tracing,
$$
\diver(E_{X})=\frac{1}{2}\nabla R_{X} \quad \Longleftrightarrow \quad \diver(\ricc_{X})=\nabla R_{X}\,.
$$
A simple computation now shows that
$$
 \diver\sq{E_{X}\owedge g}=0 \quad \Longleftrightarrow\quad (\ricc_{X})_{ij,k} -(\ricc_{X})_{ik,j} =0 \,,
$$
i.e. $\ricc_{X}$ is a Codazzi tensor.

We prove now the second equivalence. Assume that $\ricc_{X}$ is a Codazzi tensor. Then, by definition, we have
\begin{equation}\label{eq1234}
(\ricc_{X})_{ij,k}=(\ricc_{X})_{ik,j} \quad \Longleftrightarrow\quad R_{ij,k}+\frac{1}{2}\big((X_{ijk}+X_{jik}\big)=R_{ik,j}+\frac{1}{2}\big((X_{ikj}+X_{kij}\big)\,.
\end{equation}
In particular, tracing with respect to $i,j$, we deduce that
\begin{align*}
R_{k}&=\big(X_{ktt}+X_{tkt}\big)-2X_{ttk} \\
&= \big(X_{ktt}+X_{tkt}\big)-2X_{tkt}+2X_{t}R_{tk}\\
&= 2X_{t}R_{tk}+\big(X_{ktt}-X_{tkt}\big)\\
&= 2X_{t}R_{tk}+\mathcal{A}^{X}_{kt,t} \,,
\end{align*}
i.e.
\begin{equation}\label{eq98}
\nabla R = 2 \ricc(X,\cdot) + \diver(\mathcal{A}^{X}) \,.
\end{equation}
Moreover, going back to \eqref{eq1234}, one has
$$
R_{ij,k}-R_{ik,j}=\frac{1}{2}\big(X_{ikj}-X_{ijk}\big)+\frac{1}{2}\big(X_{kij}-X_{jik}\big) \,.
$$
Now we have, using again the commutation rule $X_{ijk}-X_{ikj}=X_{t}R_{tijk}$ and Bianchi identities
$$
R_{ij,k}-R_{ik,j}=C_{ijk}+\frac{1}{2(n-1)}\big(R_{k}\delta_{ij}-R_{j}\delta_{ik}\big) \,,
$$
$$
\frac{1}{2}\big(X_{ikj}-X_{ijk}\big) = \frac{1}{2}X_{t}R_{tikj}
$$
and
$$
\frac{1}{2}\big(X_{kij}-X_{jik}\big) = \frac{1}{2}\big(X_{kji}-X_{jki}\big)+X_{t}R_{tikj}\,.
$$
Thus
$$
C_{ijk}+\frac{1}{2(n-1)}\big(R_{k}\delta_{ij}-R_{j}\delta_{ik}\big) = X_{t}R_{tikj} + \frac{1}{2}\mathcal{A}^{X}_{kj,i}\,.
$$
Inserting in the previous relation the decomposition of the curvature tensor and equation \eqref{eq98}, we obtain
$$
C_{ijk}+ X_{t}W_{tikj} =D^{X}_{ijk}\,,
$$
since $D^{X}$ can be written using $\mathcal{A}^{X}$ as follows
\begin{align*}
  D^X_{ijk} &= \frac{1}{n-2}\pa{X_kR_{ij}-X_jR_{ik}}+\frac{1}{(n-1)(n-2)}\pa{X_tR_{tk}\delta_{ij}-X_tR_{tj}\delta_{ik}}\\ \nonumber
 & -\frac{R}{(n-1)(n-2)}\pa{X_k\delta_{ij}-X_j\delta_{ik}}\\ \nonumber &+\frac 12 \mathcal{A}^{X}_{kj,i}-\frac{1}{2(n-1)}\big(\mathcal{A}^{X}_{kt,t}\delta_{ij}-\mathcal{A}^{X}_{jt,t}\delta_{ik}\big) \,.
\end{align*}

Equation \eqref{def_Cotton_comp_Weyl} immediately implies
  \begin{align*}
      \ricc_{X} \text{ is a Codazzi tensor } \quad\Longleftrightarrow\quad \begin{cases}
         W_{tijk,t} = \pa{\frac{n-3}{n-2}}\Big(X_{t}W_{tijk} - D^{X}_{ijk}\Big) \\ \nabla R = 2 \ricc(X,\cdot) + \diver(\mathcal{A}^{X})
      \end{cases}
  \end{align*}
From the equivalence
 $$
 (M,g,X)\in\HCX \quad\Longleftrightarrow \quad \ricc_{X} \text{ is a Codazzi tensor }
  $$
and the fact that compact Ricci solitons are gradient, it follows that all the results concerning compact $\HCf$ metrics in Section \ref{sec-main} can be extended to the nongradient setting, defining the class $\HCXl$ in the natural way.

 \subsection*{The class $\YX$}
  In analogy with the gradient case, a simple computation shows that
   \begin{align*}
      (M,g,f)\in \YX \quad \Longleftrightarrow \quad \diver\pa{\ricc_{X}-R_{X} g}=0 \,.
        \end{align*}
Moreover, we can prove the following obstruction result which extend to the nongradient setting Corollary \ref{cor-kazsphere}.

\begin{proposition} Let $(M,g_{0})$ be a compact Riemannian manifold and $X\in\mathfrak{X}(M)$ be a non-Killing conformal vector field on $(M,g_{0})$. Then, there are no conformal metrics $g\in[g_{0}]$ such that $(M,g,X)\in\YX$.
\end{proposition}
\begin{proof}  Let $g\in[g_{0}]$. By the conformal invariance of equation \eqref{eq-cvf}, we have that $X$ is also a conformal vector field for $(M,g)$. Assume that $(M,g,X)\in\YX$, i.e.
$$
\nabla R = 2 \ricc(X,\cdot) + \diver(\mathcal{A}^{X})
$$
where $\diver(\mathcal{A}^{X})_{i}=\mathcal{A}^{X}_{ij,j}=X_{ij,j}-X_{ji,j}$. By Kazdan-Warner identity, we have
\begin{equation}\label{eq-12345}
\int_{M}\ricc(X,X)\,dV + \frac{1}{2}\int_{M}\langle \diver(\mathcal{A}^{X}), X\rangle\,dV \,=\, 0 \,.
\end{equation}
Integrating Bochner formula in Lemma \ref{lem-boch} and using the conformal vector field equation, one has
$$
\int_{M}\ricc(X,X)\,dV = \int_{M}|\nabla X|^{2}\,dV+\frac{n-2}{n}\int_{M}|\diver(X)|^{2}\,dV \,.
$$
On the other hand
\begin{align*}
\frac{1}{2}\int_{M}\langle \diver(\mathcal{A}^{X}), X\rangle\,dV &= -\frac{1}{2}\int_{M}(X_{ij}-X_{ji})X_{ij} \\
&= -\frac{1}{2}\int_{M}|\nabla X|^{2}\,dV + \frac{1}{2}\int_{M}X_{ij}X_{ji}\,dV \\
&= -\int_{M}|\nabla X|^{2}\,dV + \frac{1}{n}\int_{M}|\diver(X)|^{2}\,dV \,.
\end{align*}
Putting these identity in \eqref{eq-12345}, we obtain $\diver(X)=0$. Thus $X$ must be a Killing vector field, and this contradicts the assumption.
\end{proof}

\

\section{Final remarks and open problems}

To conclude, we present a short list of comments and open questions, which could be the subject of further investigations.

\begin{itemize}

\item[1.] In section \ref{sec-yf} we showed an obstruction to problem (B), that is a possible generalization of the {Yamabe problem} related to the class $\Yf$. A simple computation using the conformal changes of the scalar and Ricci curvature (see e.g. \cite{besse}) shows that $(M,\widetilde{g},f)=(M, e^{2u}g,f) \in \Yf$ if and only if the function $u$ solves the PDE
    \begin{align}\label{YamabefPDE}
      \nabla \Delta u &+ (n-2)\nabla^{2}u(\nabla u,\cdot)-\Big(2\Delta u + (n-2)|\nabla u|^{2} - \frac{1}{n-1}R\Big) \nabla u  -\frac{1}{2(n-1)}\nabla R  \\ \nonumber
     &= -\ricc(\nabla f, \cdot) +\frac{n-2}{n-1}\nabla^{2}u(\nabla f,\cdot) + \frac{1}{n-1}\Big(\Delta u +(n-2)|\nabla u|^{2}\Big)\nabla f
- \frac{n-2}{n-1}\langle \nabla u,\nabla f\rangle \nabla u. \,  
    \end{align}
Note that, commuting the first term, this equation is \emph{second} order problem with respect to the gradient of the conformal factor $u$. Now we ask: there exist sufficient conditions on $\nabla f$ to ensure that (B) has a positive answer, or, equivalently, to ensure the existence of solution of \eqref{YamabefPDE}? Are there other obstructions to the latter, different from Corollary \ref{cor-kazsphere}? As far as the problem (A) is concerned, are there any obstructions at all?
    Clearly, all the previous questions could also be asked for the class $\YX$.

\item[2.] In Sections \ref{sec-hcf} and \ref{sec-yf} we constructed some examples of $\HCf$ and $\Yf$ metrics, respectively, using warped products. Can we construct other examples, apart from gradient Ricci solitons $\Ef$, possibly ``non-warped''? Can we construct examples of $\HCfl$?
    Moreover, what can we say in the nongradient cases $\HCX$ and $\YX$? Since compact Ricci solitons $\EX$ are gradient, it would be interesting to construct a compact example of $\HCX$ metric, with $X$ ``genuinely'' nongradient (that is, not of the form $X=\nabla f+Y$, where $Y$ is a Killing vector field). Compare also with 5. below.

\item[3.] In the positive (sectional) curvature case we have seen, in Proposition \ref{pro-sec}, that the class $\HCfl$ coincide with the one of gradient Ricci solitons $\Ef$. Are there other characterizations? What is the role of the so-called Hamilton identity $R+\abs{\nabla f}^2-2\lambda f=C$,  $C\in \mathds{R}$, which is valid for gradient Ricci solitons?

\item[4.] Inspired by the classification results for Ricci solitons, we could study, for instance, the following problems:
\begin{itemize}
  \item[a.] If $M$ is compact and $(M, g, f) \in \HCfl$, with $\lambda\leq 0$, is it true that $(M, g)\in \E$?
  \item[b.] If $M$ is three dimensional, compact and $(M, g, f) \in \HCfl$, with $\lambda> 0$, is it true that $(M, g)$ is isometric, up to quotients, to the round sphere $\SS^3$? More generally, can we classify complete $3$-dimensional manifolds which belong to the class $\HCfl$ with $\lambda>0$?
  \item[c.] If $(M, g)$ is compact and locally conformally flat,  and $(M, g, f) \in \HCfl$, is it true that $(M, g)\in \SF$? Note that, since in section \ref{sec-hcf} we constructed locally conformally flat examples of $\HCf$ metrics,  for the latter class this result is clearly false.
  \item[d.] Gradient Ricci solitons can be classified by imposing conditions on the Weyl tensor which are weaker than local conformal flatness (e.g. harmonic Weyl curvature \cite{fergar}, Bach flatness \cite{caoche2}, higher order vanishing conditions \cite{catmasmondiv}). Can we prove similar results for $\HCf$ metrics?
\end{itemize}

\item[5.] As we saw in Section \ref{sec-nongrad}, compact Ricci solitons $\EX$ are gradient. What can we say about the classes $\HCX$, $\HCXl$ and $\YX$ in the compact case? Are there natural geometric conditions ensuring the ``gradientness'' for these classes (even in the noncompact setting)?

%
%

\end{itemize}

\

\

\begin{ackn} The authors are members of the Gruppo Nazionale per l'Analisi Matematica, la Probabilit\`{a} e le loro Applicazioni (GNAMPA) of the Istituto Nazionale di Alta Matematica (INdAM) and they are supported by GNAMPA project ``Strutture di tipo Einstein e Analisi Geometrica su variet\`a Riemanniane e Lorentziane''.
\end{ackn}

\

\

\bibliographystyle{abbrv}

\bibliography{WHC}

\

\end{document}